\documentclass[a4paper,11pt,oneside,fleqn]{article}

\usepackage[mathscr]{eucal}
\usepackage{amsmath,amssymb,amsthm}
\usepackage{hyperref}
\hypersetup{colorlinks, linkcolor=blue, citecolor=blue, urlcolor=blue}

\allowdisplaybreaks
\makeatletter
\long\def\@makecaption#1#2{%
  \vskip\abovecaptionskip\footnotesize
  \sbox\@tempboxa{#1. #2}%
  \ifdim \wd\@tempboxa >\hsize
    #1. #2\par
  \else
    \global \@minipagefalse
    \hb@xt@\hsize{\hfil\box\@tempboxa\hfil}%
  \fi
  \vskip\belowcaptionskip}
\makeatother

\newcommand{\todo}[1][\null]{\ensuremath{\clubsuit}}

\newcommand{\noprint}[1]{}
\newcommand{\checked}[1][\null]{\ensuremath{\boldsymbol{\surd}}}

\newcommand{\p}{\partial}
\newcommand{\const}{\mathop{\rm const}\nolimits}
\newcommand{\rk}{\mathop{\rm rk}}

\newcommand{\ri}{\mathfrak r}

\newtheorem{theorem}{Theorem}
\newtheorem{lemma}[theorem]{Lemma}
\newtheorem{corollary}[theorem]{Corollary}
\newtheorem{proposition}[theorem]{Proposition}
\newtheorem*{problem*}{Problem}
{\theoremstyle{definition}

\newtheorem{remark}[theorem]{Remark}
\newtheorem*{remark*}{Remark}
}

\begin{document}

\par\noindent {\LARGE\bf
Extended symmetry analysis\\ of an isothermal no-slip drift flux model
\par}

\vspace{3mm}\par\noindent {\large Stanislav Opanasenko$^{\dag\ddag}$, Alexander Bihlo$^\dag$, Roman O.\ Popovych$^{\ddag\S\natural}$ and Artur Sergyeyev$^\natural$
}

\vspace{3mm}\par\noindent {\it
$^{\dag}$~Department of Mathematics and Statistics, Memorial University of Newfoundland,\\
$\phantom{^{\dag}}$~St.\ John's (NL) A1C 5S7, Canada\par
}
\vspace{2mm}\par\noindent {\it
$^\ddag$~Institute of Mathematics of NAS of Ukraine, 3 Tereshchenkivska Str., 01024 Kyiv, Ukraine\par
}
\vspace{2mm}\par\noindent {\it
$^{\S}$~Fakult\"at f\"ur Mathematik, Universit\"at Wien, Oskar-Morgenstern-Platz 1, 1090 Wien, Austria%
}
\vspace{2mm}\par\noindent {\it
$^{\natural}$~Mathematical Institute, Silesian University in Opava, Na Rybn\'\i{}\v{c}ku 1, 746 01 Opava,\\
$\phantom{^{\natural}}$~Czech Republic
}

\vspace{2mm}\par\noindent
\textup{E-mail:} sopanasenko@mun.ca, abihlo@mun.ca, rop@imath.kiev.ua, artur.sergyeyev@math.slu.cz\!
\par

\vspace{7mm}\par\noindent\hspace*{10mm}\parbox{140mm}{\small
We perform extended group analysis for a system of differential equations modeling an isothermal no-slip drift flux.
The maximal Lie invariance algebra of this system is proved to be infinite-dimensional.
We also find the complete point symmetry group of this system, including discrete symmetries,
using the megaideal-based version of the algebraic method.
Optimal lists of one- and two-dimensional subalgebras of the maximal Lie invariance algebra in question
are constructed and employed for obtaining reductions of the system under study.
Since this system contains a subsystem of two equations that involves only two of three dependent variables,
we also perform group analysis of this subsystem.
The latter can be linearized by a composition of a fiber-preserving point transformation with
a two-dimensional hodograph transformation to the Klein--Gordon equation.
We also employ both the linearization and the generalized hodograph method
for constructing the general solution of the entire system under study.
We find {\em inter alia} genuinely generalized symmetries for this system and present the connection between them
and the Lie symmetries of the subsystem we mentioned earlier.
Hydrodynamic conservation laws and their generalizations are also constructed.
\looseness=-1
}\par\vspace{3mm}

\noprint{



Highlights

We perform extended group analysis for a model of an isothermal no-slip drift flux.

Its maximal Lie invariance algebra is proved to be infinite-dimensional.

Its complete point symmetry group is found using an original algebraic method.

The set of its local solutions is completely described via linearizing its subsystem.

First-order generalized symmetries and hydrodynamic conservation laws are computed.

}

\section{Introduction}\label{sec:IDFMIntroduction}

In the course of solving problems in physics one often faces systems of first-order quasilinear
differential equations, that is, first-order systems which are linear in the first-order
derivatives of the dependent variables but whose coefficients at these derivatives may in general depend on
the dependent and independent variables. Such systems frequently occur in acoustics, fluid mechanics,
gas and shock dynamics~\cite{Whitham1999}, and for the case of two independent variables have the general form
\begin{gather}\label{qls}
\sum_{j=1}^n A^{ij}\frac{\p u^j}{\p t}+\sum_{j=1}^n B^{ij}\frac{\p u^j}{\p x}+C^i=0,\quad i=1,\dots,n,
\end{gather}
where the $n\times n$ matrices $A=(A^{ij})$, $B=(B^{ij})$ and the $n$-component vector $C=(C^i)$
are functions of independent variables $(t,x)$ and dependent variables $(u^1,\dots,u^n)$ but not of the derivatives of the latter.
Such systems and their natural generalizations to the case of more than two independent variables are known as
(translation-noninvariant nonhomogeneous) hydrodynamic-type systems and are a subject of intense research,
see for example~\cite{BlaszakSergyeyev2009,BurdeSergyeyev2013,GrundlandHariton2008,GrundlandHuard2007,GSW, OS2013, Pavlov2003, AS17, AS18}
and references therein.

An important class of such systems is given by evolutionary translation-invariant systems of hydrodynamic type in two independent variables,
for which $A$ is the $n\times n$ unit matrix, the vector~$C$ vanishes,
and the matrix~$B$ depends only on dependent variables.
For the sake of brevity in what follows we shall refer to such systems just as to the hydrodynamic-type systems.\looseness=-1

\looseness=-2
If a hydrodynamic-type system can be diagonalized by a change~$\ri=\ri(u)$ of dependent variables alone,
that is, $\tilde B^{ij}=0$ for the new matrix~$\tilde B$ if~$i\ne j$
and $V^i:=\tilde B^{ii}$ are eigenvalues of the matrix $B$,
then the new dependent variables are called \emph{Riemann invariants} of this system,
and the eigenvalues~$V^1$, \dots, $V^n$ are commonly referred to as the characteristic velocities of the system, cf.\ e.g.\ \cite{RozhdestvenskiiJanenko1983}.
Note that in general hydrodynamic-type systems in more than two dependent variables are not diagonalizable.
A diagonalizable hydrodynamic-type system is called genuinely nonlinear if~$V^i_i\neq 0$ for all $i=1,\dots,n$,
and linearly degenerate if all of these inequalities fail.
Here and below, unless otherwise explicitly stated, the indices~$i$, $j$ and~$k$ run from~1 to~$n$,
and a function subscript like~$i$ denotes the derivative with respect to the $i$th Riemann invariant.

\looseness=-1
In the theory of hydrodynamic-type systems there exists an integrability criterion
involving the generalized hodograph transformation~\cite{Tsarev1985,Tsarev1991}.
This criterion states that a diagonalizable strictly hyperbolic
(strict hyperbolicity means that all characteristic velocities~$V^i$ are real and distinct)
hydrodynamic-type system is integrable in the sense presented below
if and only if the condition
\[
\p_i\frac{V^k_j}{V^j-V^k}=\p_j\frac{V^k_i}{V^i-V^k}
\]
holds for all $i\ne k\ne j$ (no sum over repeated indices here). Such hydrodynamic-type systems are called
\textit{semi-Hamiltonian}~\cite{Tsarev1985,Tsarev1991}
(see e.g.\ \cite[p.~60]{DubrovinNovikov1989} and~\cite{El2014} for further details).
Given a diagonal hydrodynamic-type system \[\ri^i_t+V^i\ri^i_x=0,\]
where $\ri=(\ri^1,\dots,\ri^n)$ is the tuple of its Riemann invariants
and $V=(V^1, \dots, V^n)$ is the \mbox{associated} tuple of characteristic velocities,
the generalized hodograph transformation allows one to locally represent the
general solution of this system (except for the solutions for which~$\ri^i_x=0$ for some~$i$)
in the form
\begin{gather}\label{eq:GeneralizedHodographAnsatz}
x-V^i(\ri)t=W^i(\ri),
\end{gather}
where $W=(W^1,\dots,W^n)$ is the general solution
of the system
\[
\frac{W^i_j}{W^j-W^i}=\frac{V^i_j}{V^j-V^i},\quad i\ne j,\qquad
\]
with the nondegeneracy condition $\det(V^i_jt+W^i_j)\ne0$ (again no sum over repeated indices here),
which guarantees that the ansatz~\eqref{eq:GeneralizedHodographAnsatz} is locally solvable with respect to~$\ri$.

In the rest of the present paper we shall deal with a three-component ($n=3$) hydrodynamic-type system arising in the study of the two-phase flow phenomena.
The problem in question is of great importance
in physics thanks to its applications in several rapidly
developing branches such as nuclear power and chemical
industry~\cite{EvjeFjelde2002,YunParkChung1993}.
In particular, the accurate prediction
of void fraction in the sub-channel under two-phase flow is of fundamental significance.

\looseness=-1
Unfortunately, this problem is quite challenging and therefore various simplified models of two-phase flow phenomena were
developed. One of these is the drift flux model introduced in~\cite{ZuberFindlay1965},
which allows to describe the mixing motion rather than the individual phases.
It was thoroughly studied in~\cite{EvjeFlatten2005,EvjeFlatten2007,EvjeKarlsen2008},
where several sub-models of the drift flux model were found, and, in particular,
the concept of the slip function was introduced. The model with no-slip condition
\begin{subequations}\label{eq:IDFModel}
\begin{gather}
\rho^1_t+u\rho^1_x+u_x\rho^1=0,\label{eq:IDFMInitialEquation1}\\
\rho^2_t+ u\rho^2_x+u_x\rho^2=0,\label{eq:IDFMInitialEquation2}\\
(\rho^1+\rho^2)(u_t+uu_x)+a^2(\rho^1_x+\rho^2_x)=0,\label{eq:IDFMInitialEquation3}
\end{gather}
\end{subequations}
where $u=u(t,x)$ is the common velocity in both phases,
$\rho^1=\rho^1(t,x)$ and $\rho^2=\rho^2(t,x)$ are densities of liquids (or of liquid and gas),
and $a$ is a constant depending on both phases,
was considered in~\cite{Banda2010}.
In~\cite{SekharSatapathy2016} an attempt at performing the group analysis for this system was made.
Unfortunately, the said work contains a number of inaccuracies, including an incorrect computation of the maximal
Lie invariance algebra as well as some mistakes in the classification of one-dimensional and two-dimensional subalgebras of the latter.
Also, the system~\eqref{eq:IDFModel} has many nice additional properties as it is of hydrodynamic type,
and the goal of the present paper is to revisit and to significantly extend the group analysis of \eqref{eq:IDFModel}.

To this end it is convenient to simplify the initial model~\eqref{eq:IDFModel} by an appropriate change of variables.
First of all, using a simultaneous rescaling of $x$ and $u$ we can set $a=1$ provided that $a$ is positive,
which is justified from the physical point of view.
Introducing the new dependent variables $v=\ln(\rho^1+\rho^2)$ and
$w=\rho^1/\rho^2$ instead of $\rho^1$ and $\rho^2$,
we rewrite the system~\eqref{eq:IDFModel} as the system~$\mathcal S$ which reads
\begin{subequations}\label{eq:IDFMModelForGroupAnalysis}
\begin{gather}
u_t+uu_x+v_x=0,\label{eq:IDFMEquation1}\\
v_t+uv_x+u_x=0,\label{eq:IDFMEquation2}\\
w_t+uw_x=0.\label{eq:IDFMEquation3}
\end{gather}
\end{subequations}
The system~\eqref{eq:IDFMModelForGroupAnalysis} is obviously of hydrodynamic type.
Furthermore, diagonalizing the matrix~$B$ of the representation (\ref{qls}) for the
system~\eqref{eq:IDFMModelForGroupAnalysis} we map the system~\eqref{eq:IDFMModelForGroupAnalysis}
to the system
\begin{subequations}\label{eq:IDFMDiagonalizedSystem}
\begin{gather}
\ri^1_t+(\ri^1+\ri^2+1)\ri^1_x=0,\label{eq:IDFMDEquation1}\\
\ri^2_t+(\ri^1+\ri^2-1)\ri^2_x=0,\label{eq:IDFMDEquation2}\\
\ri^3_t+(\ri^1+\ri^2)\ri^3_x=0\label{eq:IDFMDEquation3}
\end{gather}
\end{subequations}
using the change of dependent variables~$\ri^1=\frac12(u+v)$, $\ri^2=\frac12(u-v)$, $\ri^3=w$.
Thus, $\ri^1$, $\ri^2$ and~$\ri^3$ are the Riemann invariants of the system~\eqref{eq:IDFMModelForGroupAnalysis},
and the Riemann invariants of the initial system~\eqref{eq:IDFModel} are
\[
\ri^1=\frac{u+\ln(\rho^1+\rho^2)}2,\quad \ri^2=\frac{u-\ln(\rho^1+\rho^2)}2,\quad \ri^3=\frac{\rho^1}{\rho^2},
\]
and hence the expressions for the initial dependent variables $(u,\rho^1,\rho^2)$ in terms of the Riemann invariants read
\[
u=\ri^1+\ri^2,\quad
\rho^1=\frac{\ri^3e^{\ri^1-\ri^2}}{\ri^3+1},\quad
\rho^2=\frac{e^{\ri^1-\ri^2}}{\ri^3+1}.
\]
We also readily see that the characteristic velocities of the system~\eqref{eq:IDFMDiagonalizedSystem} are
\begin{gather}\label{eq:IDFMCharacteristicSpeeds}
V^1=\ri^1+\ri^2+1,\quad V^2=\ri^1+\ri^2-1,\quad V^3=\ri^1+\ri^2.
\end{gather}

The system~\eqref{eq:IDFMDiagonalizedSystem} is not genuinely nonlinear as~$V^3_3=0$.
It is strictly hyperbolic, diagonalizable and semi-Hamiltonian, so the generalized hodograph transformation
can be applied here. Nevertheless, it also makes sense to extend the scope and to
study this system within the framework of group analysis.
Note that the subsystem of the first two equations of~\eqref{eq:IDFMDiagonalizedSystem}
coincides with the diagonalized form of the system describing one-dimensional isentropic gas flows with constant sound speed
\cite[Section~2.2.7, Eq.~(16)]{RozhdestvenskiiJanenko1983}.

Throughout the text we switch between the forms~\eqref{eq:IDFMModelForGroupAnalysis}
and~\eqref{eq:IDFMDiagonalizedSystem} of the system~$\mathcal S$.
It is often more convenient to use the diagonalized form for computation,
although many results are more concisely expressed in terms of the variables~$u$, $v$ and~$w$.

The rest of the paper is organized as follows.
In Section~\ref{sec:IDFMLieSymmetries} we compute the maximal Lie invariance algebra of the system~\eqref{eq:IDFMModelForGroupAnalysis}.
Pushing forward the Lie symmetry vector fields by relevant point transformations,
we also present the corresponding algebras for both the initial system~\eqref{eq:IDFModel}
and its diagonalized form~\eqref{eq:IDFMDiagonalizedSystem}.
Section~\ref{sec:IDFMMethodsForCompleteGroupSymmetries} is devoted to some background material
on how to compute the complete point symmetry group,
which contains both Lie symmetries and discrete point symmetries, for a system of differential equations.
The actual computation of the complete point symmetry group of the system~\eqref{eq:IDFMModelForGroupAnalysis}
is presented in Section~\ref{sec:IDFMCompletePointSymmetryGroup}.
The classification of one- and two-dimensional subalgebras
of the maximal Lie invariance algebra of the system~\eqref{eq:IDFMModelForGroupAnalysis}
is presented in Section~\ref{sec:IDFMClassificationSubalgebras}.
Section~\ref{sec:IDFMReductions} contains  the results on Lie reductions based on the optimal lists of inequivalent subalgebras from the previous section.
In Section~\ref{sec:IDFMEssentialSubsystem} we perform group analysis
of the essential subsystem~\eqref{eq:IDFMEquation1}--\eqref{eq:IDFMEquation2} of the system~\eqref{eq:IDFMModelForGroupAnalysis}.
Employing the fact that this subsystem
can be linearized by a two-dimensional hodograph transformation to the telegraph equation,
in Section~\ref{sec:IDFMSolutionsViaEssentialSubsystem} we describe \emph{all} (local) solutions of the system~\eqref{eq:IDFMModelForGroupAnalysis},
where regular solutions are expressed in terms of solutions of the telegraph equation.
In Section~\ref{sec:IDFMGeneralizedHodographMethod} we apply the generalized hodograph transformation
for describing the general solution of the diagonalized version~\eqref{eq:IDFMDiagonalizedSystem} of the system~$\mathcal S$.
First-order generalized symmetries of the system~\eqref{eq:IDFMDiagonalizedSystem} and their generalizations
are computed in Section~\ref{sec:IDFM1stOrderGeneralizedSymmetries}.
Hydrodynamic conservation laws for the model under study and their generalizations are given in Section~\ref{sec:IDFMHydrodynamicCLs}.
In the final Section~\ref{sec:IDFMConclusions} we summarize the results of the paper and touch upon some avenues of future research.

\section{Lie symmetries}\label{sec:IDFMLieSymmetries}

In order to compute the maximal Lie invariance algebra of a given system of differential equations
we employ the infinitesimal method~\cite{Olver1993,Ovsiannikov1982}.

The infinitesimal generators of one-parameter Lie symmetry groups for the system~$\mathcal S$
are defined as $Q=\tau\p_t+\xi\p_x+\eta\p_u+\theta\p_v+\zeta\p_w$,
where the components~$\tau$, $\xi$, $\eta$, $\theta$ and $\zeta$ depend on~$t$, $x$, $u$, $v$, and~$w$.
The infinitesimal invariance criterion requires that
\begin{gather}\label{eq:IDFMInfinitesimalInvCriterion}
Q^{(1)}\left(u_t+uu_x+v_x\right)|_{\mathcal S}=0,\quad
Q^{(1)}\left(v_t+uv_x+u_x\right)|_{\mathcal S}=0,\quad
Q^{(1)}\left(w_t+uw_x\right)|_{\mathcal S}=0.
\end{gather}
The first prolongation $Q^{(1)}$ of the vector field~$Q$ is given by
\[
 Q^{(1)}=Q+\eta^{(1,0)}\p_{u_t}+\eta^{(0,1)}\p_{u_x}+\theta^{(1,0)}\p_{v_t}+\theta^{(0,1)}\p_{v_x}+
 \zeta^{(1,0)}\p_{w_t}+\zeta^{(0,1)}\p_{w_x},
\]
where the components~$\eta^{(1,0)}$, $\eta^{(0,1)}$, $\theta^{(1,0)}$, $\theta^{(0,1)}$, $\zeta^{(1,0)}$ and~$\zeta^{(0,1)}$
of the prolonged vector field~$Q^{(1)}$ are readily derived from the general prolongation formula~\cite{Olver1993},
\begin{gather*}
\eta^\alpha  =\mathrm D^\alpha\left(\eta  -\tau u_t-\xi u_x\right)+\tau u_{\alpha+\delta_1}+\xi u_{\alpha+\delta_2},\\
\theta^\alpha=\mathrm D^\alpha\left(\theta-\tau v_t-\xi v_x\right)+\tau v_{\alpha+\delta_1}+\xi v_{\alpha+\delta_2},\\
\zeta^\alpha =\mathrm D^\alpha\left(\zeta -\tau w_t-\xi w_x\right)+\tau w_{\alpha+\delta_1}+\xi w_{\alpha+\delta_2}.
\end{gather*}
Here $\alpha=(\alpha_1,\alpha_2)$ is a multi-index, $\delta_1=(1,0)$, $\delta_2=(0,1)$,
$\mathrm D^\alpha=\mathrm D_t^{\alpha_1}\mathrm D_x^{\alpha_2}$,
$\mathrm D_t$ and $\mathrm D_x$ are the total derivative operators with respect to~$t$ and~$x$,
respectively,
\begin{gather*}
\mathrm D_t=\partial_t+\sum_\alpha(u_{\alpha+\delta_1}\partial_{u_\alpha}+v_{\alpha+\delta_1}\partial_{v_\alpha}+w_{\alpha+\delta_1}\partial_{w_\alpha}),\\
\mathrm D_x=\partial_x+\sum_\alpha(u_{\alpha+\delta_2}\partial_{u_\alpha}+v_{\alpha+\delta_2}\partial_{v_\alpha}+w_{\alpha+\delta_2}\partial_{w_\alpha})
\end{gather*}
with $u_\alpha=\p^{\alpha_1+\alpha_2}u/\p t^{\alpha_1}\p x^{\alpha_2}$, etc.
Thus, the infinitesimal invariance criterion implies that
\begin{gather*}
\eta^{(1,0)}  +u\eta^{(0,1)}   +\eta u_x + \theta^{(0,1)}=0,\\
\theta^{(1,0)}+u\theta^{(0,1)} +\eta v_x + \eta^{(0,1)}=0,\\
\zeta^{(1,0)}+u\zeta^{(0,1)} +\eta w_x  =0
\end{gather*}
when substituting $u_t=-uu_x-v_x$, $v_t=-uv_x-u_x$ and $w_t=-uw_x$.
Splitting with respect to the parametric derivatives $u_x$, $v_x$ and~$w_x$
results in the system of determining equations for the components of Lie symmetry vector fields,
\begin{gather*}
\zeta_t=\zeta_x=\zeta_u=\zeta_v=\tau_x=\tau_u=\tau_v=\tau_w=\eta_u=\eta_v=\eta_w=\theta_u=\theta_v=\theta_w=0,\\
\xi_x=\tau_t,\quad\xi_t=\eta,\quad\eta_t+u\eta_x+\theta_x=0,\quad\theta_t+u\theta_x+\eta_x=0.
\end{gather*}
The general solution of this system is
\begin{gather*}
\tau=\xi^1t+\tau^0,\quad\xi=\eta^0t+\xi^1x+\xi^0,\quad\eta=\eta^0,\quad\theta=\theta^0,\quad\zeta=\Omega(w),
\end{gather*}
where $\tau^0$, $\xi^0$, $\xi^1$, $\eta^0$, $\eta^1$ and $\theta^0$ are arbitrary constants
and $\Omega$ is an arbitrary smooth function of its argument.
This proves the following theorem.

\begin{theorem}
The maximal Lie invariance algebra~$\mathfrak g$ of the system~$\mathcal S$ is infinite-dimensional
and spanned by the vector fields
\begin{gather}\label{eq:IDFMSymmetryOperatorsSystem3}
\mathcal D=t\p_t+x\p_x,\quad
\mathcal G=t\p_x+\p_u,\quad
\mathcal P^t=\p_t,\quad
\mathcal P^x=\p_x,\quad
\mathcal P^v=\p_v,\quad
\mathcal W(\Omega)=\Omega(w)\p_w,
\end{gather}
where $\Omega$ runs through the set of smooth functions of $w$.
\end{theorem}
\begin{remark}\label{IDFMTrueInvarianceAlgebra}
The systems~\eqref{eq:IDFModel} and~\eqref{eq:IDFMModelForGroupAnalysis} are related
by the point transformation with $t$, $x$ and $u$ unchanged, and
$\rho^1=we^v/(w+1)$, $\rho^2=e^v/(w+1)$.
This transformation pushes forward the algebra~$\mathfrak g$ to
the maximal Lie invariance algebra~$\tilde{\mathfrak g}$ of the initial system~\eqref{eq:IDFModel},
where~$\tilde{\mathfrak g}=\langle\tilde{\mathcal D},\tilde{\mathcal G},\tilde{\mathcal P}^t,\tilde{\mathcal P}^x,
\tilde{\mathcal P}^v,\tilde{\mathcal W}(\tilde\Omega)\rangle$,
where $\tilde{\mathcal D}$, $\tilde{\mathcal G}$, $\tilde{\mathcal P}^t$, $\tilde{\mathcal P}^x$
are formally of the same form as their counterparts in the algebra~$\mathfrak g$,
\[
\tilde{\mathcal P}^v=\rho^1\p_{\rho^1}+\rho^2\p_{\rho^2},\quad
\tilde{\mathcal W}(\tilde\Omega)=\rho^2\tilde\Omega\left(\frac{\rho^1}{\rho^2}\right)(\p_{\rho^1}-\p_{\rho^2}),
\]
and $\tilde\Omega$ runs through the set of the smooth functions of its argument.
Note that the infinite-dimensional part~$\langle\tilde{\mathcal W}(\tilde\Omega)\rangle$
of the algebra~$\tilde{\mathfrak g}$ was missed in~\cite{SekharSatapathy2016}.
\end{remark}

\begin{remark}\label{rem:IDFMDiagonalization}
Likewise, the maximal Lie invariance algebra of the system~\eqref{eq:IDFMDiagonalizedSystem} is spanned by the vector fields
\begin{gather*}
\hat{\mathcal D}=t\p_t+x\p_x,\quad
\hat{\mathcal G}=2t\p_x+\p_{\ri^1}+\p_{\ri^2},\quad 
\hat{\mathcal P}^t=\p_t,\quad
\hat{\mathcal P}^x=\p_x,\quad
\hat{\mathcal P}^v=\p_{\ri^1}-\p_{\ri^2},\\ 
\hat{\mathcal W}(\Omega)=\Omega(\ri^3)\p_{\ri^3},
\end{gather*}
where $\Omega$ runs through the set of smooth functions of $\ri^3$,
and we also rescaled some of the spanning elements.
\end{remark}

\section{Methods of finding the complete point symmetry group}\label{sec:IDFMMethodsForCompleteGroupSymmetries}

Using the Lie infinitesimal method for a system of differential equations~$\mathcal L$
with subsequent generation of finite transformations
allows one to construct the Lie symmetry group of~$\mathcal L$,
which consists of continuous symmetry transformations of the system~$\mathcal L$
and is the identity component of the complete point symmetry group of this system.
At the same time, discrete symmetry transformations are also of interest for applications.
If the Lie symmetry group of~$\mathcal L$ is known,
then finding discrete symmetry transformations of~$\mathcal L$
is equivalent to the construction of the complete point symmetry group of~$\mathcal L$.
There are two main methods for computing complete point symmetry groups of systems of differential equations in the literature,
the direct method~\cite{PopovychBihlo2010,VaneevaJohnpillaiPopovychSophocleous2007,VaneevaPopovychSophocleous2012}
and the algebraic method~\cite{BihloCardosoPopovych2015,BihloPopovych2011,Hydon2000,Hydon2000Book,PopovychCardoso2015}.
Although the latter in general gives only a part of restrictions on the form of point symmetry transformations
and thus also involves computations within the framework of the direct method at the final step of the respective procedure,
it is usually more convenient since, in comparison with using the direct method alone,
the computations are less cumbersome.

The direct method is based on the definition of a (finite) point symmetry transformation and is the most universal one.
The application of this method results in a system of PDEs known as the determining equations which are, in general,
nonlinear and strongly coupled and thus quite difficult to solve. Several methods were
developed to improve the direct method, and one of those improved methods is the algebraic one.


The algebraic method for finding the complete point symmetry
group of a system of differential equations was suggested by Hydon~\cite{Hydon2000,Hydon2000Book}.
The underlying idea is that each point symmetry transformation~$\mathcal T$
of a system of differential equations~$\mathcal L$ induces an automorphism
of the maximal Lie invariance algebra~$\mathfrak g$ of~$\mathcal L$.
If the algebra~$\mathfrak g$ is finite-dimensional with $\dim\mathfrak g=n>0$, then the above means that
\[
\mathcal T_* e_j=\sum\limits_{i=1}^na_{ij}e_i, \quad j=1,\dots, n,
\]
where $\mathcal T_*$ is the pushforward of vector fields induced by~$\mathcal T$,
and $(a_{ij})_{i,j=1}^n$ is the matrix of an automorphism of~$\mathfrak g$
in the chosen basis~$(e_1,\dots,e_n)$.
Computing the automorphism group $\text{Aut}(\mathfrak g)$ of~$\mathfrak g$,
one obtains a set of restrictions on the form of the matrix $(a_{ij})_{i,j=1}^n$.
Expanding the above condition for $\mathcal T_*$ gives, under the assumption that $\mathfrak g\ne\{0\}$,
gives constraints for the transformation~$\mathcal T$.
These constraints are to be subsequently employed within the framework of the direct method.
\looseness=-1

Be that as it may, the computation of the entire automorphism group $\text{Aut}(\mathfrak g)$
can still be challenging, especially if the maximal Lie invariance algebra~$\mathfrak g$ is infinite-dimensional.
This is why another version of the algebraic method,
which involves the notion of megaideals~\cite{PopovychBoykoNesterenkoLutfullin2003},
also known as fully characteristics ideals~\cite{HilgertNeeb2011},
was developed in~\cite{BihloCardosoPopovych2015,BihloPopovych2011,PopovychCardoso2015}.
A megaideal~$\mathfrak i$ of a Lie algebra~$\mathfrak g$ is a subspace of~$\mathfrak g$
that is invariant under any automorphism of~$\mathfrak g$.
Provided~$\mathfrak g$ is finite-dimensional with $\dim\mathfrak g=n$ and possesses a megaideal~$\mathfrak i$
spanned by the first~$k$, $k<n$, basis elements, i.e., $\mathfrak i=\langle e_1,\cdots,e_k\rangle$,
the matrix~$(a_{ij})_{i,j=1}^n$ of any automorphism of~$\mathfrak g$
is of block structure with $a_{ij}=0$ for $i>k$ and $j\leqslant k$.
In other words, the megaideal hierarchy of a finite-dimensional Lie algebra
is directly related to the block structure of matrices of the automorphisms of this Lie algebra.
This observation can also be applied to infinite-dimensional Lie algebras with finite-dimensional megaideals.

Simple tools for constructing megaideals were presented in~\cite{BihloPopovych2011,PopovychBoykoNesterenkoLutfullin2003,PopovychCardoso2015}.
First of all, both the improper subalgebras of a Lie algebra~$\mathfrak g$
(the zero subalgebra and~$\mathfrak g$ itself) are (improper) megaideals of~$\mathfrak g$.
Moreover, sums, intersections and Lie products of megaideals are megaideals,
megaideals of megaideals of~$\mathfrak g$ are megaideals of~$\mathfrak g$,
all elements of the derived series, the ascending and the descending central series of~$\mathfrak g$,
in particular, the center and the derivatives of~$\mathfrak g$,
as well as the radical and the nilradical are megaideals of~$\mathfrak g$.
Recall that the $n$th derivative~$\mathfrak g^{(n)}$
of the Lie algebra~$\mathfrak g$ is defined by
$\mathfrak g^{(n)}=[\mathfrak g^{(n-1)},\mathfrak g^{(n-1)}]$ for $n\geqslant1$ with $\mathfrak g^{(0)}=\mathfrak g$
and the $n$th power~$\mathfrak g^n$ of the algebra~$\mathfrak g$ is
$\mathfrak g^n=[\mathfrak g^{n-1},\mathfrak g]$ for $n\geqslant2$ with
$\mathfrak g^1=\mathfrak g$.
One more, less obvious, way of obtaining new megaideals from known ones is given by the following assertion.

\begin{proposition}\label{prop:IDFMMegaideal}
If  $\mathfrak i_0$, $\mathfrak i_1$ and $\mathfrak i_2$ are megaideals of~$\mathfrak g$,
then the set
$\mathfrak s=\{z\in\mathfrak i_0\colon[\langle z \rangle,\mathfrak i_1]\subseteq\mathfrak i_2\}$
is also a megaideal of~$\mathfrak g$.
\end{proposition}

\looseness=-1
Within the megaideal-based version of the algebraic method,
one uses the condition $\mathcal T_*\mathfrak i=\mathfrak i$
for each megaideal~$\mathfrak i$ from a certain set of known megaideals of
the maximal Lie invariance algebra~$\mathfrak g$ of the system~$\mathcal L$.
Megaideals that are sums of other megaideals give constraints that are consequences
of constraints that are derived from the consideration of the summands.
This is why such decomposed megaideals should be neglected in the course of the computation.
Upon having derived the constraints for a point symmetry transformation~$\mathcal T$
implied by the megaideal invariance,
one completes the computation of the complete point symmetry group with the direct method.

\section{Complete point symmetry group}\label{sec:IDFMCompletePointSymmetryGroup}

We compute the complete point symmetry group of the system~\eqref{eq:IDFMModelForGroupAnalysis}
using  the megaideal-based version of the algebraic method.

The nonzero commutation relations among generating elements~\eqref{eq:IDFMSymmetryOperatorsSystem3}
of the maximal Lie invariance algebra~$\mathfrak g$ of the system~$\mathcal S$ are exhausted by
\begin{gather*}
[\mathcal P^t,\mathcal D]=\mathcal P^t,\quad
[\mathcal P^x,\mathcal D]=\mathcal P^x,\quad
[\mathcal P^t,\mathcal G]=\mathcal P^x,\quad
[\mathcal W(\Omega^1),\mathcal W(\Omega^2)]=\mathcal W(\Omega^1\Omega^2_w-\Omega^2\Omega^1_w).
\end{gather*}
Therefore, the algebra~$\mathfrak g$ is the direct sum
of its finite-dimensional and infinite-dimensional parts,
$\mathfrak g=\langle\mathcal D,\mathcal G,\mathcal P^t, \mathcal P^x, \mathcal P^v\rangle\oplus\langle\mathcal W(\Omega)\rangle$.
Moreover, the finite-dimensional part can be split into a direct sum as well,
so $\mathfrak g=\langle\mathcal D,\mathcal G,\mathcal P^t, \mathcal P^x
\rangle\oplus\langle\mathcal P^v\rangle\oplus\langle\mathcal W(\Omega)\rangle$.

We now construct a list of megaideals for the algebra~$\mathfrak g$.
First, the derivatives of~$\mathfrak g$ are megaideals of~$\mathfrak g$, so
$\mathfrak g'=\langle\mathcal P^t, \mathcal P^x, \mathcal W(\Omega)\rangle$ and
$\mathfrak g''=\langle\mathcal W(\Omega)\rangle$ are megaideals,
with $\mathfrak g^{(i)}=\mathfrak g''$ for $i\geqslant2$.
The center $\mathcal Z(\mathfrak g)=\langle\mathcal P^v\rangle$ of the algebra~$\mathfrak g$ is also its megaideal.

\begin{lemma}\label{lem:IDFMRadical}
The radical~$\mathfrak r$ of~$\mathfrak g$ coincides with the finite-dimensional part of~$\mathfrak g$,
\[\mathfrak r=\langle\mathcal D,\mathcal G,\mathcal P^t, \mathcal P^x, \mathcal P^v\rangle.\]
\end{lemma}

\begin{proof}
We temporarily denote by~$\mathfrak s$ the finite-dimensional part of~$\mathfrak g$,
$\mathfrak s=\langle\mathcal D,\mathcal G,\mathcal P^t, \mathcal P^x, \mathcal P^v\rangle$.
The subspace~$\mathfrak s$ is an ideal of~$\mathfrak g$,
which is solvable since $\mathfrak s''=\{0\}$.
Therefore, it is contained in the radical $\mathfrak r$ of~$\mathfrak g$, $\mathfrak s\subseteq\mathfrak r$.
If an ideal of~$\mathfrak g$ contains a vector field~$\mathcal W(\Omega^0)$
for a particular nonvanishing $\Omega^0=\Omega^0(w)$ substituted for $\Omega$,
then it contains the entire infinite-dimensional part $\langle\mathcal W(\Omega)\rangle$
and hence it is not solvable.
This means that $\mathfrak r\cap\langle\mathcal W(\Omega)\rangle=\{0\}$.
Therefore, $\mathfrak r=\mathfrak s$.
\end{proof}

Thus, we find a Levi decomposition of the infinite-dimensional algebra~$\mathfrak g$,
$\mathfrak g=\mathfrak r\oplus\mathfrak g''$, where
$\mathfrak r$ is the (finite-dimensional) radical and
$\mathfrak g''$ is an (infinite-dimensional) simple subalgebra,
which is also an (mega)ideal of~$\mathfrak g$ but contains no proper subideals.

\begin{lemma}
The nilradical~$\mathfrak n$ of~$\mathfrak g$ is spanned by the vector fields
$\mathcal G$, $\mathcal P^t$, $\mathcal P^x$ and~$\mathcal P^v$,
\[
\mathfrak n=\langle\mathcal G,\mathcal P^t,\mathcal P^x,\mathcal P^v\rangle.
\]
\end{lemma}

\begin{proof}\looseness=-1
The nilradical of~$\mathfrak g$ is contained in the radical~$\mathfrak r$ of~$\mathfrak g$.
We temporarily denote by~$\mathfrak s$ the span of the vector fields
$\mathcal G$, $\mathcal P^t$, $\mathcal P^x$ and~$\mathcal P^v$,
$\mathfrak s=\langle\mathcal G,\mathcal P^t,\mathcal P^x,\mathcal P^v\rangle$.
The subspace~$\mathfrak s$ is an ideal of~$\mathfrak g$, and it is nilpotent since $\mathfrak s^2=\{0\}$.
Moreover, $\mathfrak s$ is the maximal nilpotent ideal of~$\mathfrak g$
since the only subspace of~$\mathfrak r$ properly containing $\mathfrak s$
is the radical~$\mathfrak r$ itself, which is not nilpotent. Thus, $\mathfrak n=\mathfrak s$.
\end{proof}

\begin{corollary}
The derivatives $\mathfrak r'=\langle\mathcal P^t,\mathcal P^x\rangle$
and $\mathfrak n'=\langle\mathcal P^x\rangle$ of the
radical~$\mathfrak r$ and the nilradical~$\mathfrak n$ of~$\mathfrak g$, respectively, are megaideals of~$\mathfrak g$.
\end{corollary}

\begin{corollary}
The ideal~$\mathfrak m_1=\langle\mathcal G,\mathcal P^x,\mathcal P^v\rangle$ of the algebra~$\mathfrak g$
is its megaideal.
\end{corollary}

\begin{proof}
This is a simple consequence of Proposition~\ref{prop:IDFMMegaideal} for
$\mathfrak i_0=\mathfrak r$, $\mathfrak i_1=\mathfrak r$ and
$\mathfrak i_2=\mathfrak n'$.
\end{proof}

The nilradical~$\mathfrak n$ is not essential for the megaideal-based version
of the algebraic method since it is the sum of other megaideals of~$\mathfrak g$,
$\mathfrak n=\mathfrak m_1+\mathfrak r'$.
As a result, for finding the complete point symmetry group of the system~\eqref{eq:IDFMModelForGroupAnalysis}
with the megaideal-based version of the algebraic method
we use the following list of megaideals of the algebra~$\mathfrak g$:
\begin{equation}\label{eq:IDFMModelMegaideaList}
\langle\mathcal D,\mathcal G,\mathcal P^t, \mathcal P^x, \mathcal P^v\rangle,\quad
\langle\mathcal G,\mathcal P^x,\mathcal P^v\rangle,\quad
\langle\mathcal P^t, \mathcal P^x\rangle,\quad \langle\mathcal P^x\rangle,\quad \langle\mathcal P^v\rangle,\quad
\langle\mathcal W(\Omega)\rangle.
\end{equation}

\begin{theorem}\label{thm:IDFMSymmetryGroup}
The complete point symmetry group~$G$ of the modified no-slip isothermal drift flux model~\eqref{eq:IDFMModelForGroupAnalysis}
consists of the transformations
\begin{gather}\label{eq:IDFMSymmetryGroup}
\begin{split}
&\tilde t=T^1t+T^0,\quad \tilde x=T^1x+T^1U^0t+X^0,\\
&\tilde u=u+U^0, \quad \tilde v=v+V^0,\quad \tilde w=W(w),
\end{split}
\end{gather}
where~$T^0$, $T^1$, $X^0$, $U^0$ and~$V^0$ are arbitrary constants with $T^1\ne0$
and~$W$ runs through the set of smooth functions of~$w$ with $W_w\ne0$.
\end{theorem}

\begin{proof}
The general form of a point symmetry transformation for the system~\eqref{eq:IDFMModelForGroupAnalysis} is
\begin{gather*}
\mathcal T\colon\ (\tilde t,\tilde x,\tilde u,\tilde v,\tilde w)=(T,X,U,V,W),
\end{gather*}
where~$T$, $X$, $U$, $V$ and~$W$ are functions of~$t$, $x$, $u$, $v$ and~$w$ with nonvanishing Jacobian.
To obtain constraints for a point symmetry transformation~$\mathcal T$,
we push forward each of the Lie symmetry generators~\eqref{eq:IDFMSymmetryOperatorsSystem3}, $Q$,
by this transformation and use the invariance, with respect to the pushforward~$\mathcal T_*$,
of the minimal megaideal from the list~\eqref{eq:IDFMModelMegaideaList} that contains~$Q$.
This leads to the following conditions:
\begin{gather}\label{eq:IDFMPushforwardAction}
\begin{split}
&\mathcal T_*\mathcal D=a_{11}(\tilde t\p_{\tilde t}+\tilde x\p_{\tilde x})+a_{21}(\tilde t\p_{\tilde x}+
\p_{\tilde u})+a_{31}\p_{\tilde t}+a_{41}\p_{\tilde x}+a_{51}\p_{\tilde v},\\
&\mathcal T_*\mathcal G=a_{22}(\tilde t\p_{\tilde x}+\p_{\tilde u})+
a_{42}\p_{\tilde x}+a_{52}\p_{\tilde v},\\
&\mathcal T_*\mathcal P^t=T_t\p_{\tilde t}+X_t\p_{\tilde x}+U_t\p_{\tilde u}
+V_t\p_{\tilde v}+W_t\p_{\tilde w}=a_{33}\p_{\tilde t}+a_{43}\p_{\tilde x},\\
&\mathcal T_*\mathcal P^x=T_x\p_{\tilde t}+X_x\p_{\tilde x}+U_x\p_{\tilde u}
+V_x\p_{\tilde v}+W_x\p_{\tilde w}=a_{44}\p_{\tilde x},\\
&\mathcal T_*\mathcal P^v=T_v\p_{\tilde t}+X_v\p_{\tilde x}+U_v\p_{\tilde u}
+V_v\p_{\tilde v}+W_v\p_{\tilde w}=a_{55}\p_{\tilde v},\\
&\mathcal T_*\mathcal W(\Omega)=\Omega(T_w\p_{\tilde t}+X_w\p_{\tilde x}+U_w\p_{\tilde u}
+V_w\p_{\tilde v}+W_w\p_{\tilde w})=\tilde\Omega^\Omega\p_{\tilde w},
\end{split}
\end{gather}
where all~$a$'s are constants
and $\tilde\Omega^\Omega$ is a smooth function of~$\tilde w$ depending on the parameter function $\Omega=\Omega(w)$.

We collect components of vector fields in the last four conditions of~\eqref{eq:IDFMPushforwardAction}
and then take into account the resulting equations when expanding and componentwise splitting the first two conditions.
As a result, we arrive at the following system:
\begin{gather*}
\begin{split}
&T_t=a_{33},\quad X_t=a_{43},\quad U_t=V_t=W_t=0;\quad
 T_x=U_x=V_x=W_x=0,\quad X_x=a_{44};\\
&T_v=X_v=U_v=W_v=0,\quad V_v=a_{55};\quad
 T_w=X_w=U_w=V_w=0,\quad \Omega W_w=\tilde\Omega^\Omega(W);\\
&tT_t=a_{11}T+a_{31},\quad tX_t+xX_x=a_{11}X+a_{41},\quad a_{21}=a_{51}=0;\\
&T_u=W_u=0,\quad tX_x+X_u=a_{22}T+a_{42},\quad U_u=a_{22},\quad V_u=a_{52}.
\end{split}
\end{gather*}
Equations of this system are partitioned into groups according to their source conditions in~\eqref{eq:IDFMPushforwardAction}.
The general solution of this system is given by
\begin{gather}\label{eq:IDFMMegaidealsAftermath}
\begin{split}
&T=a_{33}t-a_{31},\quad X=a_{22}a_{33}x+a_{43}t-a_{41},\\
& U=a_{22}u+U^0,\quad V=a_{55}v+a_{52}u+V^0,\quad W=W(w),
\end{split}
\end{gather}
where $U^0$ and~$V^0$ are arbitrary constants,
and additionally $a_{22}a_{33}a_{55}\ne0$, $a_{11}=1$, $a_{44}=a_{22}a_{33}$ and $a_{42}=a_{31}a_{22}$.

Now the direct method of computing complete point symmetry groups should be applied.
To this end, we apply a transformation of the form~\eqref{eq:IDFMMegaidealsAftermath}
to the system~$\mathcal S$.
To do this, the derivative operators with respect to the new independent
variables, $\p_{\tilde t}$ and~$\p_{\tilde x}$, have to be determined,
\begin{gather*}
\p_{\tilde t}=\frac{1}{a_{33}}\left(\p_t-\frac{a_{42}}{a_{22}a_{33}}\p_x \right),\quad
\p_{\tilde x}=\frac{1}{a_{22}a_{33}}\p_x.
\end{gather*}
Employing these derivative operators to express the derivatives of the new variables,
and enforcing the symmetry condition requires that $a_{42}=a_{33}U^0$, $a_{22}=a_{55}=1$, $a_{52}=0$.

Re-denoting parameter constants completes the proof of the theorem.
\end{proof}

\begin{corollary}
The modified no-slip isothermal drift flux model~\eqref{eq:IDFMModelForGroupAnalysis}
possesses two independent (up to combining with each other and with continuous symmetries)
discrete point symmetries given by the reflections
\begin{gather*}
(t,x,u,v,w)\rightarrow(-t,-x,u,v,w)
\quad\mbox{and}\quad
(t,x,u,v,w)\rightarrow(t,x,u,v,-w).
\end{gather*}
\end{corollary}

\begin{remark}
As we have already mentioned in Section~\ref{sec:IDFMMethodsForCompleteGroupSymmetries},
the automorphism-based version of the algebraic method for finding the complete point symmetry group
requires knowing the automorphism group of the corresponding maximal invariance algebra.
For the system~$\mathcal S$, this algebra is infinite-dimensional,
and therefore the computation of its automorphism group is complicated.
However, it could be useful to look for
the automorphism group $\mathrm{Aut}(\mathfrak r)$ of the radical~$\mathfrak r$
which coincides, in view of Lemma~\ref{lem:IDFMRadical}, with the finite-dimensional part of the algebra~$\mathfrak g$.
Note that the algebra~$\mathfrak r$ is isomorphic to the algebra $A^0_{4,8}\oplus A_1$
from the classification list of five-dimensional real Lie algebras
presented in~\cite{Mubarakzyanov1963b,Mubarakzyanov1963a}.
In the basis $(\mathcal D,\mathcal G,\mathcal P^t, \mathcal P^x, \mathcal P^v)$,
the automorphism group $\mathrm{Aut}(\mathfrak r)$
group can be identified with the matrix group constituted by the matrices of the form
\[
\left(
\begin{matrix}
1      &0            &0      &0            &0\\
0      &b_{22}       &0      &0            &0\\
b_{31} &0            &b_{33} &0            &0\\
b_{41} &b_{31}b_{22} &b_{43} &b_{22}b_{33} &0\\
b_{51} &b_{52}       &0      &0            &b_{55}
\end{matrix}\right) \qquad \mbox{with}\qquad b_{22}b_{33}b_{55}\ne0.
\]
The knowledge of~$\mathrm{Aut}(\mathfrak r)$ allows us to set constraints on the constants~$a$'s
in the conditions~\eqref{eq:IDFMPushforwardAction} before analyzing these conditions,
\[
a_{11}=1,\quad a_{22}a_{33}a_{55}\ne0,\quad a_{21}=0,\quad a_{44}=a_{22}a_{33},\quad a_{42}=a_{31}a_{22}.
\]
The structure of the automorphism matrices obviously implies
that there is one more megaideal of~$\mathfrak r$ and, therefore, of~$\mathfrak g$,
$\mathfrak m_2=\langle \mathcal D, \mathcal P^t,\mathcal P^x,\mathcal P^v\rangle$.
This completes the description of megaideals of the algebra~$\mathfrak g$.
More specifically, the megaideals of~$\mathfrak g$ are exhausted by
the essential megaideals~$\mathfrak n'$, $\mathcal Z(\mathfrak g)$, $\mathfrak r'$, $\mathfrak m_1$, $\mathfrak m_2$ and $\mathfrak g''$
and their sums.
The megaideal~$\mathfrak m_2$ cannot be found
using the means presented in Section~\ref{sec:IDFMMethodsForCompleteGroupSymmetries}.
The presence of this megaideal explains
the first of the above constraints on the constants~$a$'s
whilst the rest of these constraints cannot be obtained using the megaideal-based version of algebraic method,
being, at the same time, a direct consequence of its automorphism-based counterpart applied to~$\mathfrak r$.
In fact, this discussion shows a possibility for combining
the automorphism- and megaideal-based versions of algebraic method
for finding the complete point symmetry groups of systems of differential equations.
\end{remark}

\begin{remark}
There is one constraint for the constants~$a$'s, $a_{51}=0$, among those derived from the conditions~\eqref{eq:IDFMPushforwardAction}
that cannot be obtained from the structure of automorphism matrices of the radical~$\mathfrak r$.
This means that there exist automorphisms of the entire algebra~$\mathfrak g$
that are not induced by point transformations of~$(t,x,u,v,w)$.
\end{remark}

\section{Classification of subalgebras}\label{sec:IDFMClassificationSubalgebras}

In order to efficiently perform group-invariant reductions
for finding exact solutions of the system~\eqref{eq:IDFMModelForGroupAnalysis},
it is necessary to construct an optimal list of $G$-inequivalent one- and two-dimensional subalgebras
of the maximal Lie invariance algebra~$\mathfrak g$ admitted by this system~\cite{Olver1993}.
Since the algebra~$\mathfrak g$ is infinite-dimensional,
for constructing the adjoint representation of~$G$ on~$\mathfrak g$
we use the method based on computing pushforwards of vector fields in~$\mathfrak g$
by transformations in~$G$ \cite{CardosoBihloPopovych2011,SzatmariBihlo2012}.

\pagebreak

Any transformation~$\mathcal T$ from~$G$ can be represented as a composition of elementary symmetry transformations,
$\mathcal T=\mathscr D(T^1)\mathscr P(T^0)\mathscr P(X^0)\mathscr P(V^0)\mathscr G(U^0)\mathscr W(W)$~with
\begin{gather*}
\mathscr P^t(T^0)  \colon (\tilde t,\tilde x,\tilde u,\tilde v,\tilde w)=(t+T^0,x,u,v,w),\quad
\mathscr P^x(X^0)  \colon (\tilde t,\tilde x,\tilde u,\tilde v,\tilde w)=(t,x+X^0,u,v,w), \\
\mathscr P^v(V^0)  \colon (\tilde t,\tilde x,\tilde u,\tilde v,\tilde w)=(t,x,u,v+V^0,w),\quad
\mathscr D(T^1)    \colon (\tilde t,\tilde x,\tilde u,\tilde v,\tilde w)=(T^1t,T^1x,u,v,w),\\
\mathscr G(U^0)    \colon (\tilde t,\tilde x,\tilde u,\tilde v,\tilde w)=(t,x+U^0t,u+U^0,v,w),\,
\mathscr W(W) \colon (\tilde t,\tilde x,\tilde u,\tilde v,\tilde w)=(t,x,u,v,W(w)),
\end{gather*}
where, as in Theorem~\ref{thm:IDFMSymmetryGroup},
$T^0$, $T^1$, $X^0$, $U^0$ and~$V^0$ are arbitrary constants with $T^1\ne0$
and~$W(w)$ runs through the set of smooth functions of~$w$ with $W_w\ne0$.
The nonidentity actions of pushforwards of elementary symmetry transformations
on generating elements of~$\mathfrak g$ are exhausted by the following:
\begin{gather}\label{eq:IDFMPushforwardActiononGeneratingElements}
\begin{split}
&\mathscr P^t_*(T^0)\mathcal D=\mathcal D - T^0\mathcal P^t,\quad
\mathscr P^t_*(T^0)\mathcal G=\mathcal G - T^0\mathcal P^x,\quad
\mathscr P^x_*(X^0)\mathcal D=\mathcal D - X^0\mathcal P^x,\\
&\mathscr D_*(T^1)\mathcal P^t=T^1\mathcal P^t,\quad
\mathscr D_*(T^1)\mathcal P^x=T^1\mathcal P^x,\quad
\mathscr G_*(U^0)\mathcal P^t=\mathcal P^t + U^0\mathcal P^x,\\
&\mathscr W_*(W)\mathcal W(\Omega)=\mathcal W(\tilde\Omega),
\end{split}
\end{gather}
where the function $\tilde\Omega=\tilde\Omega(\tilde w)$ is related to~$\Omega=\Omega(w)$ by $\tilde\Omega(W(w))=W_w(w)\Omega(w)$.

\begin{theorem}\label{thm:IDFMOneDimensionalSubalgebras}
An optimal list of one-dimensional subalgebras of the maximal Lie invariance al\-geb\-ra~$\mathfrak g$
of the system~\eqref{eq:IDFMModelForGroupAnalysis} is exhausted by the subalgebras
\begin{gather}\label{thm:IDFMSubalgebras}
\begin{split}
&\langle\mathcal D+a\mathcal G+b\mathcal P^v+\mathcal W(\delta_1)\rangle,\quad
\langle\mathcal G+\delta_2\mathcal P^t+b\mathcal P^v+\mathcal W(\delta_1)\rangle,\quad
\langle\mathcal P^t+\delta_2\mathcal P^v+\mathcal W(\delta_1)\rangle,\\
&\langle\mathcal P^x+\delta_2\mathcal P^v+\mathcal W(\delta_1)\rangle,\quad
\langle \mathcal P^v+\mathcal W(\delta_1)\rangle,\quad
\langle \mathcal W(1)\rangle,
\end{split}
\end{gather}
where $\delta_1,\delta_2\in\{0,1\}$, and $a$ and~$b$ are arbitrary constants.
\end{theorem}
\begin{proof}
We start with the most general form of a basis vector field
of a one-dimensional subalgebra of the algebra~$\mathfrak g$,
\begin{gather*}
Q=a_1\mathcal D+a_2\mathcal G+a_3\mathcal P^t+a_4\mathcal P^x+a_5\mathcal P^v+\mathcal W(\Omega),
\end{gather*}
where $a_i$, $i=1,\dots,5$, are constants, $\Omega$ is a smooth function of $w$,
and at least one of these quantities does not vanish.
We simplify~$Q$ using the adjoint action of~$G$ and scaling of the entire~$Q$.
In all cases when $\Omega\ne0$, the parameter function $\Omega$ can be set to $1$
by applying the pushforward~$\mathscr W_*\big(\int 1/\Omega\,{\rm d}w\big)$.
In other words, we can always suppose $\Omega(w)=\delta_1\in\{0,1\}$.
Moreover, if $a_5\ne0$, then the summand with~$\mathcal P^v$ cannot be deleted using the adjoint action of~$G$.
\looseness=-1

We order the elements spanning~$\mathfrak g$ in the following way:
$\mathcal D\succ\mathcal G\succ\mathcal P^t\succ\mathcal P^x\succ\mathcal P^v\succ\mathcal W(\Omega)$.
This partial order is justified by the hierarchy of megaideals of the algebra~$\mathfrak g$.
Subsequent analysis is done recursively according to this order:
regarding the leading coefficient of~$Q$ as nonzero,
we scale this coefficient, if it is a constant, to~1 by scaling the entire vector field~$Q$,
further simplify the form of~$Q$ by the adjoint action of~$G$
and re-denote the parameters preserved in the final form of~$Q$.
As a result, we reduce the classification to the following inequivalent cases.
\looseness=-1

1.\ $a_1=1$.
Using successively the pushforwards $\mathscr P^t_*(a_3)$ and $\mathscr P^x_*(a_4-a_2a_3)$
we set $a_3=a_4=0$.
This yields the first subalgebra (more precisely, the first family of subalgebras) of the list~\eqref{thm:IDFMSubalgebras}.

2.\ $a_1=0$, $a_2=1$.
We act by $\mathcal P^t_*(a_4)$ and, if $a_3\ne0$, by $\mathscr D_*(a_3^{-1})$
to set $a_4=0$ and $a_3=1$, respectively,
which leads to the second subalgebra.

3.\ $a_1=a_2=0$, $a_3=1$.
Then we act by $\mathscr G_*(-a_4)$ to set $a_4=0$.
If $a_5\ne0$, then we can also set $a_5=1$ by acting $\mathscr D_*(a_5)$ and rescaling~$Q$.
This gives the third subalgebra.

4.\ $a_1=a_2=a_3=0$, $a_4=1$.
Upon acting by~$\mathscr D_*(a_5)$ and rescaling~$Q$, we derive the fourth subalgebra.

5.\ $a_i=0$, $i=1,\dots,4$, $a_5=1$. This yields the fifth subalgebra.

6.\ $a_i=0$, $i=1,\dots,5$.
Then $\Omega$ is nonzero, which corresponds to the last subalgebra.
\end{proof}

\begin{theorem}\label{thm:IDFMTwoDimensionalSubalgebras}
An optimal list of two-dimensional subalgebras of the maximal al\-geb\-ra
of invariance~$\mathfrak g$ of the system~\eqref{eq:IDFMModelForGroupAnalysis} is given by
\begin{gather}
\begin{split}\label{IDFMTwoDimensionalSubalgebras} 
&
 \langle\mathcal D+a\mathcal P^v+\mathcal W(\delta_1),\mathcal G+b\mathcal P^v+\mathcal W(\delta_2)\rangle,\quad
 \langle\mathcal D+a\mathcal P^v+\mathcal W(\delta_5),\mathcal P^t\rangle,
\\&
\langle\mathcal D+a\mathcal P^v+\mathcal W(w),\mathcal P^t+\mathcal W(1)\rangle,\quad
 \langle\mathcal D+a\mathcal G+b\mathcal P^v+\mathcal W(\delta_5),\mathcal P^x\rangle,
\\&
\langle\mathcal D+a\mathcal G+b\mathcal P^v+\mathcal W(w),\mathcal P^x+\mathcal W(1)\rangle,\quad
 \langle\mathcal D+a\mathcal G+\mathcal W(\delta_1),      \mathcal P^v+\mathcal W(\delta_2)\rangle,
\\&
 \langle\mathcal G+\delta_3\mathcal P^t+a\mathcal P^v+\mathcal W(\delta_1), \mathcal P^x+\delta_4\mathcal P^v+\mathcal W(\delta_2)\rangle, \quad \langle\mathcal G+\delta_5\mathcal P^t+\mathcal W(\delta_1),\mathcal P^v+\mathcal W(\delta_2)\rangle,
\\&
\langle\mathcal P^t+\delta_3\mathcal P^v+\mathcal W(\delta_1),\mathcal P^x+\delta_4\mathcal P^v+\mathcal W(\delta_2)\rangle,
\quad \langle\mathcal P^t+\mathcal W(\delta_1),                   \mathcal P^v+\mathcal W(\delta_2)\rangle,\quad
\\&
 \langle\mathcal P^x+\mathcal W(\delta_1),                   \mathcal P^v+\mathcal W(\delta_2)\rangle,\quad
 \langle\mathcal D+a\mathcal G+b\mathcal P^v+c\mathcal W(w),    \mathcal W(1)\rangle,
\\&
 \langle\mathcal G+\delta_5\mathcal P^t+b\mathcal P^v+c\mathcal W(w),\mathcal W(1)\rangle,\quad
 \langle\mathcal P^t+\delta_1\mathcal P^v+\delta_2\mathcal W(w),            \mathcal W(1)\rangle,
\\&
 \langle\mathcal P^x+\delta_1\mathcal P^v+\delta_2\mathcal W(w),            \mathcal W(1)\rangle,\quad
 \langle\mathcal P^v+c\mathcal W(w),                                 \mathcal W(1)\rangle,\quad
 \langle\mathcal W(w),                                               \mathcal W(1)\rangle,\quad
\end{split}
\end{gather}
where $\delta_1$, \dots, $\delta_4$, $a$, $b$ and~$c$ are arbitrary constants such that
one of nonzero components in each of the pairs $(\delta_1,\delta_2)$ and $(\delta_3,\delta_4)$
can be set to be equal~1, 
and $\delta_5\in\{0,1\}$.
In the seventh family of subalgebras, either $a=0$ if $\delta_4\ne0$ or $\delta_1=0$ if $\delta_2\ne0$.
\end{theorem}

\begin{proof}
We start with a two-dimensional subalgebra $\mathfrak s:=\langle Q_1,Q_2\rangle$
spanned by two linearly independent vector fields from~$\mathfrak g$ of the most general form,
\begin{gather*}
Q_1=a_1\mathcal D+a_2\mathcal G+a_3\mathcal P^t +a_4\mathcal P^x+ a_5\mathcal P^v+\mathcal W(\Omega^1),\\
Q_2=b_1\mathcal D+b_2\mathcal G+b_3\mathcal P^t +b_4\mathcal P^x+ b_5\mathcal P^v+\mathcal W(\Omega^2).
\end{gather*}
Within the classification procedure, we should take into account
the condition of closedness of~$\mathfrak s$ with respect to the Lie bracket,
$[Q_1,Q_2]\in\mathfrak s$, which we will briefly call the $\mathfrak s$-condition below.
Since $\mathfrak g=\langle\mathcal D,\mathcal G,\mathcal P^t, \mathcal P^x
\rangle\oplus\langle\mathcal P^v\rangle\oplus\langle\mathcal W(\Omega)\rangle$,
summands with $\mathcal P^v$ or $\mathcal W(\Omega)$, if they are in $Q^1$ or~$Q^2$,
cannot be deleted by the adjoint action of~$G$.
Nevertheless, we can always set $\Omega^i=1$
for nonzero~$\Omega^i$ with fixed $i\in\{1,2\}$ by the action of $\mathscr W_*\big(\int 1/\Omega^i\,{\rm d}w\big)$.
If the parameter functions~$\Omega^1$ and~$\Omega^2$ are linearly independent,
linearly recombining~$Q_1$ and~$Q_2$, which is consistent with the further classification procedure,
reduces the $\mathfrak s$-condition to $[Q_1,Q_2]\in\langle Q_2\rangle$.
Setting $\Omega^2=1$, we can thus assume $\Omega^1=c_1w+c_2$, where $c_1$ and~$c_2$ are constants.
In particular, $\Omega^2=\const$ if $Q^1$ and~$Q^2$ commute.
We can consider only $\Omega$'s of the above form from the very beginning.
If $a_1\ne0$, we can set $a_3=a_4=0$
successively using $\mathscr P^t_*(a_3)$ and $\mathscr P^x_*(a_4-a_2a_3)$.

For the sake of efficient classification of two-dimensional subalgebras of the algebra~$\mathfrak g$,
it is convenient to consider the coefficient matrix
\[A=\left(\begin{matrix}a_1&a_2&a_3&a_4&a_5\\b_1&b_2&b_3&b_4&b_5\end{matrix}\right).\]
The classification splits into two cases, $\rk A=2$ and $\rk A<2$.

\medskip\par\noindent
\textbf{I.} $\boldsymbol{\rk A=2.}$
On the set of index pairs $J:=\{(i,j)\mid 1\leqslant i<j\leqslant5\}$
we introduce the lexicographical order with the alphabet ordering $1\succ2\succ3\succ4\succ5$,
which is consistent with the ordering of vector fields spanning the algebra~$\mathfrak g$.
Then we further split the case $\rk A=2$ into subcases that are labeled by index pairs in~$J$
and each of which is singled out by the additional condition
$\rk A_{i',j'}<2$, $(i',j')\succ(i,j)$, $\rk A_{i,j}=2$ with the associated index pair $(i,j)$.
Here $A_{i,j}$ denotes the $2\times2$ matrix  composed by the $i$th and $j$th columns of $A$.
Note that $a_i=b_i=0$ if both $\rk A_{i,j}<2$ for $i<j$ and $\rk A_{j,i}<2$ for $j<i$.
In the $(i,j)$th case, up to linearly recombining~$Q_1$ and~$Q_2$
we can assume that $A_{i,j}=I$ and $b_1=\dots=b_{j'}=0$, $j'<j$.
Here and in what follows $I$ is the $2\times2$ unit matrix.
As a result, we should consider the following cases.

1.\ $A_{1,2}=I$, $a_3=a_4=0$.
The $\mathfrak s$-condition yields $b_3=b_4=0$, which results in
the first subalgebra (more precisely, the first family of subalgebras)
in the list~\eqref{IDFMTwoDimensionalSubalgebras}.

2.\ $A_{1,3}=I$, $b_2=0$, $a_4=0$.
Using $\mathscr G_*(b_4)$ we set $b_4=0$.
Then the $\mathfrak s$-condition implies $a_2=b_5=0$ and
$\Omega^2\Omega^1_w-\Omega^1\Omega^2_w=\Omega^2$.
Note that here we use the $\mathfrak s$-condition before employing the pushforward~$\mathscr W_*$.
This leads to the second and third subalgebras, depending on whether~$\Omega^2$ vanishes or not;
for the latter we should additionally act on~$\mathfrak s$ with the pushforward of a shift of~$w$.

3.\ $A_{1,4}=I$, $b_2=b_3=0$, $a_3=0$.
We employ the $\mathfrak s$-condition in the same way as in the previous case,
obtaining the fourth and fifth subalgebras.

In all the other cases with $\rk A=2$, the basis elements $Q_1$ and~$Q_2$ commute in view of the $\mathfrak s$-condition.
Hence $\Omega^1\Omega^2_w-\Omega^2\Omega^1_w=0$, and we can simultaneously set
both the coefficients~$\Omega^1$ and~$\Omega^2$ to be constants.

4.\ $A_{1,5}=I$, $b_2=b_3=b_4=0$, $a_3=a_4=0$.
Thus, we derive the sixth subalgebra.

5.\ $A_{2,3}=I$, $a_1=b_1=0$.
$[Q_1,Q_2]=\mathcal P^x+\mathcal W(\Omega^1\Omega^2_w-\Omega^2\Omega^1_w)\notin\langle Q_1,Q_2\rangle$ and hence
$\langle Q_1,Q_2\rangle$ is not a subalgebra.

6.\ $A_{2,4}=I$, $a_1=b_1=b_3=0$.
One of the coefficients~$a_3$ or~$b_5$ can be scaled to 1 (if nonzero)
using~$\mathscr D_*(a_3^{-1})$ for~$a_3$
or~$\mathscr D_*(b_5)$ and rescaling of~$Q^2$ for $b_5$.
We obtain the seventh algebra.
The specific gauge for parameters of this case from the theorem is set by composing
the pushforward $\mathscr P^t_*(c)$ with the replacement $Q^1$ by $Q^1+cQ^2$,
where $c=-a/\delta_4$ or $c=-\delta_1/\delta_2$, respectively.

7.\ $A_{2,5}=I$ and $a_1=b_1=b_3=b_4=0$.
We act on~$\mathfrak s$ by $\mathcal P^t_*(a_4)$ and, if $a_3\ne0$, by $\mathscr D_*(a_3^{-1})$
to set $a_4=0$ and $a_3=1$, respectively.
This gives the eighth subalgebra.

8.\ $A_{3,4}=I$, $a_i=b_i=0$, $i=1,2$.
One of $a_5$ and $b_5$ (if nonzero) can be scaled to 1
acting on~$\mathfrak s$ with $\mathscr D_*(c)$
and replacing the basis element~$Q_i$ by $c^{-1}Q_i$,
where $c=a_5$, $i=1$ and $c=b_5$, $i=2$, respectively.
As a result, we have the ninth subalgebra.

9.\ $A_{3,5}=I$, $a_i=b_i=0$, $i=1,2$, $b_4=0$.
Setting $a_4=0$ with $\mathcal G_*(-a_4)$ gives the tenth subalgebra.

10.\ $A_{4,5}=I$, $a_i=b_i=0$, $i=1,2,3$.
This case is associated with the eleventh subalgebra.

\medskip\par\noindent
\textbf{II.} $\boldsymbol{\rk A<2.}$
Linearly combining~$Q_1$ and~$Q_2$ in this case, we can set $b_1=\dots=b_5=0$.
Since then $\Omega^2\ne0$, we can set $\Omega^2=1$
using the pushforward~$\mathscr W_*\big(\int 1/\Omega^2\,{\rm d}w\big)$.
Up to linearly recombining~$Q_1$ and~$Q_2$,
the $\mathfrak s$-condition implies~$\Omega^1=cw$ for a constant~$c$.
The further classification of inequivalent forms for~$Q_1$, which gives the rest of listed subalgebras,
is similar to the classification of one-dimensional subalgebras,
and thus we omit~it.
\end{proof}

\section{Reductions}\label{sec:IDFMReductions}

Although in Section~\ref{sec:IDFMSolutionsViaEssentialSubsystem} below
we describe the entire set of local solutions of the system~$\mathcal S$ in implicit form
in terms of solutions of the telegraph equation,
it is instructive to find invariant solutions as well.
We employ the optimal lists of one- and two-dimensional subalgebras of the algebra~$\mathfrak g$
from Theorems~\ref{thm:IDFMOneDimensionalSubalgebras} and~\ref{thm:IDFMTwoDimensionalSubalgebras},
respectively, to obtain Lie reductions associated with them.
Lie ansatzes constructed using one-dimensional subalgebras reduce
the system~$\mathcal S$ with the two independent variables~$(t,x)$
to systems of ODEs.
The last two families of subalgebras from Theorem~\ref{thm:IDFMOneDimensionalSubalgebras}
are not relevant to the framework of Lie reductions
since the transversality condition does not hold for them.

Lie ansatzes constructed using two-dimensional subalgebras reduce
the system~$\mathcal S$ to systems of algebraic equations.
Note that constraints imposed by invariance with respect to two-parameter transformation groups
are too restrictive for solutions of PDEs with two independent variables.
This is why such solutions are not of much significance.
Nevertheless, we are going to perform the Lie reduction with respect to
a two-dimensional subalgebra of the algebra~$\mathfrak g$ as an example.
It is obvious that in the list~\eqref{IDFMTwoDimensionalSubalgebras}
of families of inequivalent two-dimensional subalgebras of the algebra~$\mathfrak g$,
only subalgebras from the first to fifth, and seventh and ninth families are appropriate for Lie reduction.
Subalgebras from other families, which do not satisfy the transversality condition,
can be useful for constructing partially invariant solutions~\cite{Ovsiannikov1982},
which we also illustrate by an example.

Within this section, $c_1$, $c_2$ and~$c_3$ are arbitrary constants,
and prime denotes the derivative with respect to~$\omega$.

\medskip

$\boldsymbol{1.\ \langle\mathcal D+a\mathcal G+b\mathcal P^v+\mathcal W(\delta_1)\rangle.}$
A Lie ansatz constructed with this subalgebra has the form
$u=\phi(\omega)+x/t+a$, $v=\chi(\omega)+b\ln|t|$, $w=\psi(\omega)+\delta_1\ln|t|$,
where $\omega=x/t-a\ln|t|$. The corresponding reduced system of ODEs is
\begin{gather*}
\phi'\phi+\chi'+\phi+a=0,\quad
\phi\chi'+\phi'+b+1=0,\quad
\phi\psi'+\delta_1=0.
\end{gather*}
For integrating this system, it is convenient to separately consider the singular and general cases.

A.\ $\phi^2+a\phi-b-1=0$,
i.e., $\phi=\mu$, where $\mu$ is a root of this algebraic equation.
Then $\chi=-(a+\mu)\omega+c_2$. For~$\psi$ we additionally have two cases.
If $\mu=0$, then  $\psi$ is an arbitrary function of $\omega$,
as well as $b=-1$ and $\delta_1=0$ are the only possible values of these parameters.
Otherwise, $\psi=-\delta_1\omega/\mu+c_3$.

B.\ Otherwise, $\phi^2+a\phi-b-1\ne0$. We also have $\phi^2\ne1$,
and the reduced system is equivalent to the system
$\phi'=(\phi^2+a\phi-b-1)/(1-\phi^2)$, $\chi'=(b\phi-a)/(1-\phi^2)$, $\psi'=-\delta_1/\phi$.
As a result, we construct the general solution of the reduced system in parametric form,
\begin{gather}\label{eq:IDFMExpressionsForOmegaChiPsi}
\omega=\int\frac{(1-\phi^2)\,{\rm d}\phi}{\phi^2+a\phi-b-1},\quad
\chi=\int\frac{(b\phi-a)\,{\rm d}\phi}{\phi^2+a\phi-b-1},\quad
\hat\psi=\int\frac{(1-\phi^2)\,{\rm d}\phi}{\phi(\phi^2+a\phi-b-1)},
\end{gather}
and $\psi=-\delta_1\hat\psi$.
All the three integrals are integrals of rational functions, which can be easily computed.

\medskip

$\boldsymbol{2.\ \langle\mathcal G+\delta_2\mathcal P^t+b\mathcal P^v+\mathcal W(\delta_1)\rangle.}$
Let us consider separately the cases $\delta_2=1$ and $\delta_2=0$.

A.\ $\delta_2=1$.
A Lie ansatz constructed for this value of~$\delta_2$ and the associated reduced system of ODEs are
$u=\phi(\omega)+t$, $v=\chi(\omega)+bt$, $w=\psi(\omega)+\delta_1t$, where $\omega=x-t^2/2$,
and
\begin{gather*}
\phi\phi'+\chi'+1=0,\quad
\phi\chi'+\phi'+b=0,\quad
\phi\psi'+\delta_1=0.
\end{gather*}
For $b=\varepsilon:=\pm1$, the reduced system admits the singular solution
$\phi=\varepsilon$,  $\chi=-\omega+c_2$ and $\psi=-\delta_1\varepsilon\omega+c_3$.
In the general case, we exclude $\chi'$ from the first two equations of the system
and derive the equation $\phi'=(\phi-b)/(1-\phi^2)$.
Integrating this equation and then using the obtained solution for integrating
the equations for~$\chi$ and~$\psi$,
we find the general solution of the reduced system in parametric form,
\begin{gather*}
\omega=-\frac{\phi^2}2-b\phi-(b^2-1)\ln|\phi-b|+c_1,\quad
\chi=b\phi+(b^2-1)\ln|\phi-b|+c_2,\\
\psi=\delta_1\phi+\frac{\delta_1}{b}\ln|\phi|+\frac{\delta_1}{b}(b^2-1)\ln|\phi-b|+c_3.
\end{gather*}

B.\ $\delta_2=0$.
A Lie ansatz constructed for this value of~$\delta_2$ and the associated reduced system of ODEs are
$u=\phi(\omega)+x/t$, $v=\chi(\omega)+bx/t$, $w=\psi(\omega)+\delta_1x/t$, where $\omega=t$, and
\[
\omega\phi'+\phi+b=0,\quad
\omega\chi'+b\phi+1=0,\quad
\omega\psi'+\delta_1\phi=0.
\]
This yields $\phi=c_1/\omega-b$, $\chi=(b^2-1)\ln|\omega|+c_1b/\omega+c_2$ and
$\psi=\delta_1b\ln|\omega|+c_1/\omega+c_3$.

\medskip

$\boldsymbol{3.\ \langle\mathcal P^t+\delta_2\mathcal P^v+\mathcal W(\delta_1)\rangle.}$
A Lie ansatz $u=\phi(\omega)$, $v=\chi(\omega)+\delta_2t$, $w=\psi(\omega)+\delta_1t$, where $\omega=x$,
constructed with this subalgebra reduces the system~\eqref{eq:IDFMModelForGroupAnalysis}
to the system of ODEs
\begin{gather*}
\phi\phi'+\chi'=0,\quad
\phi\chi'+\phi'+\delta_2=0,\quad
\phi\psi'+\delta_1=0.
\end{gather*}
Combining the first two equations to exclude $\chi'$,
we derive the equations $\delta_2+\phi'=\phi^2\phi'$,
which integrates to $\phi^3/3-\phi=\delta_2\omega+c_1$.

If $\delta_2=0$, then $\phi$ is just a constant satisfying the algebraic equation.
Hence $\chi$ is another constant.
If $\phi\ne0$, then $\psi$ is linear function with respect to $\omega$.
Otherwise, $\delta_1$ vanishes and $\psi(\omega)$ is an arbitrary smooth function of its argument.

If $\delta_2\ne0$, we construct the general solution of the reduced system in parametric form
\[
\omega=\frac{\phi^3}{3\delta_2}-\frac\phi{\delta_2}-\frac{c_1}{\delta_2},\quad
\chi=\frac{\phi^2}2+c_2,\quad
\psi=\frac{\delta_1}2\phi^2-\delta_1\ln|\phi|+c_3.
\]

\noprint{
\medskip

$\boldsymbol{4.\ \langle\mathcal P^x+\delta_2\mathcal P^v+\mathcal W(\delta_1)}\rangle$.
A Lie ansatz constructed using this algebra and the associated reduced system of ODEs are
$u=\phi(\omega)$, $v=\chi(\omega)+\delta_2x$, $w=\psi(\omega)+\delta_1x$, where $\omega=t$, and
\begin{gather*}
\phi'+\delta_2=0,\quad\chi'+\delta_2\phi=0,\quad
\psi'+\delta_1\phi=0.
\end{gather*}
This immediately gives $\phi=-\delta_2\omega+c_1$, $\chi=\delta_2^2\omega^2/2-c_1\delta_2\omega+c_2$
and $\psi=\delta_1\delta_2\omega^2/2-c_1\delta_1\omega+c_3$.
}

\medskip

$\boldsymbol{4.\ \langle\mathcal D+a\mathcal G+b\mathcal P^v+\mathcal W(\delta_1),\mathcal P^x\rangle.}$
The ansatz $u=a\ln|t|+c_1$, $v=b\ln|t|+c_2$, $w=\delta_1\ln|t|+c_3$
constructed with this subalgebra gives only trivial constant solutions
since the corresponding reduced system is $a=b=\delta_1=0$,
which is rather the condition for consistency of this ansatz with the system~$\mathcal S$.

\medskip

$\boldsymbol{5.\ \langle\mathcal D+a\mathcal G+b\mathcal P^v,\mathcal W(1)\rangle.}$
This subalgebra does not satisfy the transversality condition
in view of the form of the second basis vector field~$\mathcal W(1)$.
Therefore, this subalgebra cannot be used within the framework of Lie reductions.
At the same time, it can be used for finding partially invariant solutions
of the system~\eqref{eq:IDFMModelForGroupAnalysis}.
The existence of partially invariant solutions for the system~\eqref{eq:IDFMModelForGroupAnalysis}
is guaranteed by the fact that this system is only partially coupled
(see Section~\ref{sec:IDFMEssentialSubsystem} below for details).
We can reduce the essential subsystem~\eqref{eq:IDFMEquation1}--\eqref{eq:IDFMEquation2}
with a Lie ansatz constructed using the one-dimensional subalgebra of~$\mathfrak g$
spanned by the first basis vector field of the two-dimensional algebra under study.
Then we find solutions of the reduced system of ODEs,
merge them with the corresponding ansatz to construct solutions
of~\eqref{eq:IDFMEquation1}--\eqref{eq:IDFMEquation2}
and substitute the later solutions to the equation~\eqref{eq:IDFMEquation3}.
The resulting equation is a first-order linear partial differential equation
with respect to~$w$, which is integrated separately.
For constructing the required solutions
of the essential subsystem~\eqref{eq:IDFMEquation1}--\eqref{eq:IDFMEquation2},
it in fact suffices to use the expressions for~$(u,v)$ obtained
in the course of the first reduction of this section.

A.\ $u=x/t+a+\mu$, $v=-(a+\mu)x/t+(b+a^2+a\mu)\ln|t|+c_2$.
Then the function~$w$ satisfies the equation $w_t+(x/t+a+\mu)w_x=0$,
the general solution of which is $w=\psi(\varpi)$,
where~$\psi$ runs through the set of smooth functions of $\varpi=x/t-(a+\mu)\ln|t|$.

B.\ $u=\phi(\omega)+x/t+a$, $v=\chi(\omega)+b\ln|t|$, where $\omega=x/t-a\ln|t|$,
and $\phi$ and~$\chi$ are defined by~\eqref{eq:IDFMExpressionsForOmegaChiPsi}.
After changing the independent variables $(t,x)$ to $(\tau,\omega)$, where $\tau=\ln|t|$,
the equation~\eqref{eq:IDFMEquation3} with the above value of~$u$
takes the form $w_\tau+\phi(\omega)w_\omega=0$,
and thus its general solution is $w=\psi(\varpi)$,
where~$\psi$ runs through the set of smooth functions of
$\varpi=\tau-\int{\rm d}\omega/\phi=\ln|t|-\hat\psi$,
and $\hat\psi$ is defined by~\eqref{eq:IDFMExpressionsForOmegaChiPsi} as well.

\section{Group analysis of the essential subsystem}\label{sec:IDFMEssentialSubsystem}

\looseness=-1
The system of equations~\eqref{eq:IDFMEquation1}--\eqref{eq:IDFMEquation2},
which we denote for brevity by~$\mathcal S_0$, does not involve the unknown function~$w$,
i.e., the system~$\mathcal S$ is only partially coupled.
This is why we can solve the system~$\mathcal S_0$ in the first place
and substitute the obtained $u$ into the equation~\eqref{eq:IDFMEquation3},
which then becomes a linear first-order PDE with respect to~$w$.
Finding a partial solution of the equation~\eqref{eq:IDFMEquation3}
and acting on it by the transformation~$\mathscr W(W)$,
we obtain a family of solutions parameterized by an arbitrary function.
Moreover, the system~$\mathcal S_0$ is simpler for finding exact solutions than~$\mathcal S$ because it has fewer dependent variables
and admits a larger Lie symmetry algebra.

To compute Lie symmetries of the system~$\mathcal S_0$,
we again apply the Lie infinitesimal method, cf.\ Section~\ref{sec:IDFMLieSymmetries}.
We look for the generator $Q=\tau\p_t+\xi\p_x+\eta\p_u+\theta\p_v$ of a one-parameter point symmetry group,
where the components~$\tau$, $\xi$, $\eta$ and~$\theta$ depend on $t$, $x$, $u$ and $v$.
For the system~$\mathcal S_0$ and the generator~$Q$, the infinitesimal invariance criterion reads
\[
 Q^{(1)}\left(u_t+uu_x+v_x\right)|_{\mathcal S_0}=0,\qquad
 Q^{(1)}\left(v_t+uv_x+u_x\right)|_{\mathcal S_0}=0.
\]
The first prolongation $Q^{(1)}$ of the vector field~$Q$ is given by
\[
 Q^{(1)}=Q+\eta^{(1,0)}\p_{u_t}+\eta^{(0,1)}\p_{u_x}+\theta^{(1,0)}\p_{v_t}+\theta^{(0,1)}\p_{v_x},
\]
where the components~$\eta^{(1,0)}$, $\eta^{(0,1)}$, $\theta^{(1,0)}$ and~$\theta^{(0,1)}$
of the prolonged vector field~$Q^{(1)}$ are obtained from the general prolongation formula~\cite{Olver1993,Ovsiannikov1982}.
Then the infinitesimal invariance criterion implies
\begin{gather*}
\eta^{(1,0)}  +u\eta^{(0,1)}  +\eta u_x+\theta^{(0,1)}=0,\\
\theta^{(1,0)}+u\theta^{(0,1)}+\eta v_x+\eta^{(0,1)}=0,
\end{gather*}
when substituting $u_t=-uu_x-v_x$ and~$v_t=-uv_x-u_x$.
Splitting these equations with respect to the parametric derivatives~$u_x$ and~$v_x$
yields the following system of determining equations
\begin{gather}\label{eq:IDFMDeterminingSystem2}\arraycolsep=0ex
\begin{array}{ll}
 R^1:=\eta_t+u\eta_x+\theta_x=0,\quad
&R^2:=\theta_t+u\theta_x+\eta_x=0,\\[1ex]
 R^3:=\eta_u-\theta_v=0,\quad
&R^4:=\theta_u-\eta_v=0,\\[1ex]
 R^5:=\tau_t+2u\tau_x-\xi_x=0,\quad
&R^6:=\xi_t+(u^2-1)\tau_x-\eta=0,\\[1ex]
 R^7:=u\tau_u-\tau_v-\xi_u=0,\quad
&R^8:=u\tau_v-\tau_u-\xi_v=0.
\end{array}
\end{gather}
We derive differential consequences of the system~\eqref{eq:IDFMDeterminingSystem2},
which are then (mainly, implicitly) employed for simplifying successive differential consequences of this system,
\begin{gather*}
R^1_u-R^2_v-R^3_t-uR^3_x-R^4_x=\eta_x=0,\quad
R^2_u-R^1 _v-R^4_t-uR^4_x-R^3_x=\theta_x=0,\\
R^9:=R^7_v-R^8_u=\tau_{uu}-\tau_{vv}-\tau_v=0,\ \
R^{10}:=R^8_v-R^7_u+uR^9=\xi_{uu}-\xi_{vv}-\tau_u-u\tau_v=0,\\
R^5_{uu}-R^5_{vv}-R^5_v-R^9_t-2uR^9_x+R^8_x=2\tau_{xu}=0,\quad
R^5_u-R^7_x=\tau_{tu}+2\tau_x+\tau_{xv}\noprint{+u\tau_{xu}}=0,\\
R^5_v-R^8_x=\tau_{tv}+u\tau_{xv}\noprint{+\tau_{xu}}=0,\quad
R^3_u+R^4_v+R^6_{uu}-R^6_{vv}-R^{10}_t-u^2R^9_x=\tau_{tu}+2\tau_x-\tau_{xv}\noprint{+4u\tau_{xu}+u^2\tau_{xuu}+u\tau_{tv}+u^2\tau_{xv}}=0,\\
R^6_u+R^7_t=u(\tau_{tu}+2\tau_x)-\tau_{tv}-\eta_u\noprint{+u^2\tau_{xu}}=0,\quad
R^6_v+R^8_t=-\eta_v-\tau_{tu}\noprint{+u\tau_{tv}+u^2\tau_{xv}}=0,
\end{gather*}
and hence $R^1=\eta_t=0$, $R^2=\theta_t=0$,
whereas the recombination of the last five differential consequences of the above ones gives
$\tau_{tv}=0$, $\tau_{xv}=0$, $\tau_{tu}+2\tau_x=0$, $\eta_u=0$ and $\eta_v=2\tau_x$.
Differentiating the last equation with respect to~$t$ and~$x$, we get $\tau_{xx}=\tau_{tx}=0$, i.e., $\tau_x=\const$.
Finally, we obtain the consequences
\begin{gather*}
R^3=-\theta_v=0,\quad
R^4=\theta_u-2\tau_x=0,\quad
R^6_t=\xi_{tt}=0,\quad
R^6_x=\xi_{tx}=0,\quad
R^5_t=\tau_{tt}=0.
\end{gather*}
Hence, the general solution of the system~\eqref{eq:IDFMDeterminingSystem2} is
\begin{gather*}
\tau=A_4(x-2ut)+A_1t+\tau^0(u,v),\quad
\xi=A_4t(2v-u^2+1)+A_2t+A_1x+\xi^0(u,v),\\
\eta=2A_4v+A_2,\quad
\theta=2A_4u+A_3,
\end{gather*}
where $A_1$, \dots, $A_4$ are arbitrary constants
and $(\tau^0,\xi^0)$ runs through the solution set of the system $u\tau^0_u-\tau^0_v=\xi^0_u$, $u\tau^0_v-\tau^0_u=\xi^0_v$.
This proves the following assertion.

\begin{theorem}\label{thm:IDFMLieSymsOfEssSubsystem}
The maximal Lie invariance algebra~$\mathfrak g_0$ of the system~$\mathcal S_0$ is spanned by the vector fields
\begin{gather}
\begin{split}\label{eq:IDFMSymmetryOperatorsSystem2}
&\breve{\mathcal D}=t\p_t+x\p_x,\quad \breve{\mathcal G}=t\p_x+\p_u,\quad \breve{\mathcal P}(\tau^0,\xi^0)=\tau^0(u,v)\p_t+\xi^0(u,v)\p_x,\\
&\breve{\mathcal P}^v=\p_v,\quad \breve{\mathcal J}=\left(\tfrac12x-tu\right)\p_t+t\left(v-\tfrac12u^2+\tfrac12\right)\p_x+v\p_u+u\p_v,
\end{split}
\end{gather}
where $(\tau^0,\xi^0)$ runs through the solution set of the system $u\tau^0_u-\tau^0_v=\xi^0_u$, $u\tau^0_v-\tau^0_u=\xi^0_v$.
\end{theorem}

Note that in terms of Riemann invariants the system for $(\tau^0,\xi^0)$ reads
$\xi^0_1=V^2\tau^0_1$, $\xi^0_2=V^1\tau^0_2$.

\begin{remark}
It is obvious that for any~$\Omega$ the system~$\mathcal S_0$ cannot admit
a counterpart of the Lie symmetry vector field~$\mathcal W(\Omega)$ of the system~$\mathcal S$.
The algebra~$\mathfrak g/\langle\mathcal W(\Omega)\rangle$ is isomorphic to the proper
subalgebra~$\langle\breve{\mathcal D},\breve{\mathcal G},\breve{\mathcal P}(1,0),\breve{\mathcal P}(0,1),\breve{\mathcal P}^v\rangle$ of~$\mathfrak g_0$,
where the spanning vector fields correspond to $\mathcal D$, $\mathcal G$, $\mathcal P^t$, $\mathcal P^x$, $\mathcal P^v$, respectively.
\end{remark}

By virtue of the structure of the algebra~$\mathfrak g_0$, the system~$\mathcal S_0$
can be linearized by a hodograph transformation, cf.~\cite{KumeiBluman1982}.
In general, every (1+1)-dimensional system of hydrodynamic type with two
dependent variables can be linearized by a two-dimensional hodograph transformation~\cite{RozhdestvenskiiJanenko1983}.
A hodograph transformation exchanges the roles of dependent and independent variables.
For the system~$\mathcal S_0$, the pairs of independent and dependent variables are to be interchanged,
i.e., we set $y=u$, $z=v$, $p=t$, $q=x$, where
$y$ and $z$ are the new independent variables and
$p$ and $q$ are the new dependent variables.
Then differentiating the equality $p(u,v)=t$, $q(u,v)=x$ with respect to~$t$ and~$x$ using the chain rule, we obtain
the system of linear algebraic equations for the first derivatives of~$(u,v)$
\begin{gather*}\arraycolsep=0ex
\begin{array}{ll}
p_yu_t+p_zv_t=1,\qquad & p_yu_x+p_zv_x=0,\\[1ex]
q_yu_t+q_zv_t=0,\qquad & q_yu_x+q_zv_x=1.
\end{array}
\end{gather*}
When assuming the nondegeneracy condition $\Delta:=p_yq_z-p_z q_y\ne0$, which is equivalent to $u_tv_x-u_xv_t\ne0$, the solution of this system~is
\begin{gather}\label{eq:IDFMHodographSubstitution}
u_t=\frac{q_z}{\Delta},\quad
u_x=-\frac{p_z}{\Delta},\quad
v_t=-\frac{q_y}{\Delta},\quad
v_x=\frac{p_y}{\Delta}.
\end{gather}
In the new variables, the system~$\mathcal S_0$ takes the form
\begin{gather}\label{eq:IDFMPotentialTelegraphEquation}
q_z-yp_z+p_y=0,\quad -q_y+yp_y-p_z=0.
\end{gather}
As expected, the system~\eqref{eq:IDFMPotentialTelegraphEquation} coincides with
the system for $(\tau^0,\xi^0)$ from Theorem~\ref{thm:IDFMLieSymsOfEssSubsystem}.
The cross-differentiation with respect to~$y$ and~$z$
leads to a differential consequence of the system~\eqref{eq:IDFMPotentialTelegraphEquation},
which is the telegraph equation for $p$ alone,
\begin{gather}\label{eq:IDFMTelegraphEquation}
p_{zz}+p_z=p_{yy}.
\end{gather}
In other words, the system~\eqref{eq:IDFMPotentialTelegraphEquation}
is a potential system for the equation~\eqref{eq:IDFMTelegraphEquation}.
The substitution $p=e^{-z/2}\tilde p$ reduces~\eqref{eq:IDFMTelegraphEquation} to
the Klein--Gordon equation
\begin{gather}\label{eq:IDFMKleinGordonEquation}
\tilde p_{yy}=\tilde p_{zz}-\frac{\tilde p}4.
\end{gather}

The maximal Lie invariance algebra~$\mathfrak g_{\rm KG}$
of the Klein--Gordon equation~\eqref{eq:IDFMKleinGordonEquation} is  well known~\cite{FushchichNikitin1994}.
It is spanned by the vector fields
\begin{gather*}
\hat{\mathcal P}^y=\p_y,\quad \hat{\mathcal P}^z=\p_z,\quad \hat{\mathcal J}=y\p_z+z\p_y,\quad
\hat{\mathcal D}=\tilde p\p_{\tilde p},\quad \hat{\mathcal P}(\psi)=\tilde\psi(y,z)\p_{\tilde p},
\end{gather*}
where $\tilde\psi$ runs through the solution set of~\eqref{eq:IDFMKleinGordonEquation}.
The maximal Lie invariance algebra of the telegraph equation~\eqref{eq:IDFMTelegraphEquation}
is spanned by the pushforwards of these vector fields under the transformation $p=e^{-z/2}\tilde p$,
where the variables~$y$ and~$z$ are not changed,
\begin{gather*}
\p_y,\quad \p_z-\frac12p\p_p,\quad y\p_z+z\p_y-\frac12yp\p_p,\quad p\p_p,\quad \psi(y,z)\p_p,
\end{gather*}
where $\psi=e^{-z/2}\tilde\psi(y,z)$ runs through the solution set of~\eqref{eq:IDFMTelegraphEquation}.
Although the maximal Lie invariance algebras~$\mathfrak g_0$ and~$\mathfrak g_{\rm KG}$
of the system~$\mathcal S_0$ and the Klein--Gordon equation~\eqref{eq:IDFMKleinGordonEquation}
look similar and there seems to be an obvious relation among the generating elements,
$\breve{\mathcal D}\sim\hat{\mathcal D}$,
$\breve{\mathcal G}\sim\hat{\mathcal P}^y$,
$\breve{\mathcal P}(\tau^0,\xi^0)\sim\hat{\mathcal P}(e^{z/2}\tau^0)$,
$\breve{\mathcal P}^v\sim\hat{\mathcal P}^z+\frac12\hat{\mathcal D}$,
$\breve{\mathcal J}\sim\hat{\mathcal J}$,
these algebras are in fact not isomorphic.
This can be explained using the following arguments.
Since the systems~$\mathcal S_0$ and~\eqref{eq:IDFMPotentialTelegraphEquation} are related by a hodograph transformation,
the maximal Lie invariance algebra~$\tilde{\mathfrak g}_0$ of the system~\eqref{eq:IDFMPotentialTelegraphEquation}
is the pushforward of~$\mathfrak g_0$ under this transformation.
Hence the algebras~$\mathfrak g_0$ and~$\tilde{\mathfrak g}_0$ coincide up to re-denoting variables.
The transition from the system~\eqref{eq:IDFMPotentialTelegraphEquation} to the equation~\eqref{eq:IDFMTelegraphEquation}
requires the projection $(y,z,p,q)\mapsto(y,z,p)$.
At the same time, the vector field~$\breve{\mathcal P}(0,1)$ is projected to the zero vector field.
Moreover, the vector field~$\breve{\mathcal J}$ is not projectable,
and thus the commutators
$[\breve{\mathcal J},\breve{\mathcal P}(\tau^0,\xi^0)]$ and
$[\hat{\mathcal J},\hat{\mathcal P}(e^{z/2}\tau^0)]$
are not related to each other in the above way.
This is why there is no direct relation between the set of Lie solutions of the systems~$\mathcal S_0$ and
that of the equation~\eqref{eq:IDFMKleinGordonEquation}
(resp.\ of the equation~\eqref{eq:IDFMTelegraphEquation}).

\section{Solution through linearization of the essential subsystem}\label{sec:IDFMSolutionsViaEssentialSubsystem}

Since the system~$\mathcal S$ is partially coupled and the subsystem~$\mathcal S_0$ can be linearized,
we can construct an implicit representation for the general solution of the system~$\mathcal S$
in terms of the general solution of the telegraph equation~\eqref{eq:IDFMTelegraphEquation}
or, equivalently, of the Klein--Gordon equation~\eqref{eq:IDFMKleinGordonEquation}.
Consider the potential system $p_y=\Upsilon_z+\Upsilon$, $p_z=\Upsilon_y$ of the equation~\eqref{eq:IDFMTelegraphEquation}.
The ``pseudopotential'' $\Upsilon=\Upsilon(y,z)$ is in fact related
to a usual potential of the equation~\eqref{eq:IDFMTelegraphEquation}:
the function $e^z\Upsilon$ is the potential
associated with the conservation-law characteristic $e^z$ of the equation~\eqref{eq:IDFMTelegraphEquation}.
Since the equation $p_z=\Upsilon_y$ is in conserved form,
we can introduce the second-level potential $\Phi=\Phi(y,z)$
defined by the system $\Phi_y=p$, $\Phi_z=\Upsilon$.
The equation $p_y=\Upsilon_z+\Upsilon$ implies
that the potential $\Phi$ satisfies the same telegraph equation~\eqref{eq:IDFMTelegraphEquation} as~$p$,
$\Phi_{zz}+\Phi_z=\Phi_{yy}$.
Substituting the expression $p=\Phi_y$ into the system~\eqref{eq:IDFMPotentialTelegraphEquation}
and using the equation for~$\Phi$,
we obtain a compatible system $(q-y\Phi_y+\Phi_z+\Phi)_y=0$, $(q-y\Phi_y+\Phi_z+\Phi)_z=0$ for~$q$.
The potential $\Phi$ is defined up to a constant summand.
Hence we can assume that $q=y\Phi_y-\Phi_z-\Phi$.
Then $\Delta:=p_yq_z-p_z q_y=(\Phi_{yy})^2-(\Phi_{yz})^2$,
and the nondegeneracy condition $\Delta\ne0$ reduces, on the solution set of the equation for~$\Phi$,
to the condition $\Phi_{yy}\ne\Phi_{yz}$ or, equivalently, $\Phi_{yy}\ne-\Phi_{yz}$.%
\footnote{%
The overdetermined system
$\Phi_{zz}+\Phi_z=\Phi_{yy}$,
$\Phi_{yy}=\varepsilon\Phi_{yz}$ with $\varepsilon\in\{-1,1\}$ implies that
$\Phi_{yyy}=\Phi_{yzz}+\Phi_{yz}=\varepsilon\Phi_{yyz}+\varepsilon\Phi_{yy}=\Phi_{yyy}+\varepsilon\Phi_{yy}$
and thus $\Phi_{yy}=\Phi_{yz}=0$ resulting also to $\Phi_{yy}=-\varepsilon\Phi_{yz}$.
}
To find the component of the general solution for~$w$,
we write the equation~\eqref{eq:IDFMEquation3} in the new variables,
$(q_z-yp_z)w_y-(q_y-yp_y)w_z=0$, or $p_yw_y-p_zw_z=0$
when taking into account the system~\eqref{eq:IDFMPotentialTelegraphEquation}.
We substitute the expression $p=\Phi_y$ into the last equation multiplied by~$e^z$
and take into account the equation for~$\Phi$.
As a result, we derive the equation $(e^z\Phi_z)_zw_y-(e^z\Phi_z)_yw_z=0$
meaning, in view of $e^z\Phi_z\ne\const$ that $w$ is a function of $e^z\Phi_z$ only.
Re-writing the expressions obtained for $p$, $q$ and~$w$ in terms of the initial variables,
we construct the following implicit representation of all solutions of the system~$\mathcal S$ with $u_tv_x-u_xv_t\ne0$:
$t=\Phi_u$, $x=u\Phi_u-\Phi_v-\Phi$, $w=W(e^v\Phi_v)$,
where the function $\Phi=\Phi(u,v)$ runs through the solution set of the equation $\Phi_{vv}+\Phi_v=\Phi_{uu}$,
and $W$ is an arbitrary function of its argument~$e^v\Phi_v$.
This representation can also be interpreted as parametric.

The condition $u_tv_x-u_xv_t=0$ for singular solutions of the system~$\mathcal S$
is easily seen to be equivalent to the constraint $v_x^{\,2}=u_x^{\,2}$,
i.e.,\ $v_x=\varepsilon u_x$ with $\varepsilon=\pm1$.
Then the subsystem~$\mathcal S_0$ also implies $v_t=\varepsilon u_t$.
Therefore, $v=\varepsilon u+c$, where $c$ is an arbitrary constant,
and the subsystem~$\mathcal S_0$ reduces to the single equation $u_t+uu_x+\varepsilon u_x=0$.
Suppose that $u\ne\const$, which is equivalent to the condition $u_x\ne0$.
We perform the one-dimensional hodograph transformation
with $s=t$ and $y=u$ being the new independent variables and
$q=x$ and $w$ being the new dependent variables.
Deriving the expressions for first derivatives of~$u$ and~$w$ in the hodograph variables,
\[
u_x=\frac1{q_y},\quad
u_t=-\frac{q_s}{q_y},\quad
w_x=\frac{w_y}{q_y},\quad
w_t=w_s-\frac{q_s}{q_y}w_y,
\]
we represent the equations of the reduced system,
$u_t+uu_x+\varepsilon u_x=0$ and $w_t+uw_x=0$, in these variables,
which read $q_s=y+\varepsilon$ and $q_yw_s=\varepsilon w_y$.
The equation $q_s=y+\varepsilon$ integrates to $q=(y+\varepsilon)s+\varphi(y)$,
where $\varphi$ is an arbitrary function of~$y$.
The general solution of the equation $q_yw_s=\varepsilon w_y$ with respect to~$w$
is an arbitrary function of a first integral~$I$ of the ODE ${\rm d}s/{\rm d}y=-\varepsilon(s+\varphi_y)$.
To derive a nice expression for this integral,
instead of~$\varphi$ we introduce a new arbitrary function $\Theta=\Theta(y)$ such that $\varphi=e^{-\varepsilon y}\Theta_y$,
which gives $I=e^{\varepsilon y}s+\varepsilon\Theta_y-\Theta$.
Returning to the initial variables, we obtain
$x-(u+\varepsilon)t=e^{-\varepsilon u}\Theta_u$, $v=\varepsilon u+c$ and $w=W(I)$,
where $c$ is an arbitrary constant, $\varepsilon=\pm1$, $\Theta=\Theta(u)$ is an arbitrary function of~$u$
and $W$ is an arbitrary function of $I=e^{\varepsilon u}t+\varepsilon\Theta_u-\Theta$.

If $u=\const$, then also $v=\const$,
and the system~$\mathcal S$ reduces to the equation~\eqref{eq:IDFMEquation3}
whose general solution in this case takes the form $w=W(x-ut)$
with $W$ being an arbitrary function of its argument~$x-ut$.
We refer to this case as ultra-singular.

We collect all the three cases for solutions of the system~$\mathcal S$ into a single assertion.

\begin{theorem}\label{thm:IDFMCompleteSolution}
Any solution of the system~$\mathcal S$ (locally) belongs to one of the following families;
below $W$~is an arbitrary function of its argument.

\medskip\par\noindent
1. The regular family, where $u_tv_x-u_xv_t\ne0$ (the general solution):
\begin{gather}\label{eq:IDFMGenSolution}
t=\Phi_u,\quad
x=u\Phi_u-\Phi_v-\Phi,\quad
w=W(e^v\Phi_v).
\end{gather}
Here the function $\Phi=\Phi(u,v)$ runs through the solutions of the telegraph equation $\Phi_{vv}+\Phi_v=\Phi_{uu}$
with $\Phi_{uu}\ne\Phi_{uv}$.

\medskip\par\noindent
2. The singular family, where $u_tv_x-u_xv_t=0$ but $u$ and~$v$ are not constants:
\[
x-(u+\varepsilon)t=e^{-\varepsilon u}\Theta_u,\quad
v=\varepsilon u+c,\quad
w=W(e^{\varepsilon u}t+\varepsilon\Theta_u-\Theta).
\]
Here $c$ is an arbitrary constant, $\varepsilon=\pm1$,
and $\Theta=\Theta(u)$ is an arbitrary function of~$u$.

\medskip\par\noindent
3. The ultra-singular family, where $u$ and~$v$ are arbitrary constants and $w=W(x-ut)$.
\end{theorem}

In other words, the regular, singular and ultra-singular families of solutions of the system~$\mathcal S$
are associated with solutions of the subsystem~$\mathcal S_0$ with rank 2, 1 and 0, respectively;
cf.\ \cite{GrundlandHuard2006}.

\section{Solution using generalized hodograph transformation}\label{sec:IDFMGeneralizedHodographMethod}

Another approach to finding the general solution of the system~\eqref{eq:IDFModel}
is to employ the generalized hodograph method (see~\cite{Tsarev1985} and Section~\ref{sec:IDFMIntroduction})
for the diagonalized form~\eqref{eq:IDFMDiagonalizedSystem} of \eqref{eq:IDFModel}.
In this way, one employs the semi-Hamiltonian property related to the diagonalized form~\eqref{eq:IDFMDiagonalizedSystem}
instead of the partial coupling,
which emerges in both the simplified forms~\eqref{eq:IDFMModelForGroupAnalysis} and~\eqref{eq:IDFMDiagonalizedSystem}.
For the system~\eqref{eq:IDFMDiagonalizedSystem},
the tuple of the parameter functions~$W=(W^1,W^2,W^3)$ in the ansatz~\eqref{eq:GeneralizedHodographAnsatz} runs through
the solution set of the overdetermined linear systems of first-order partial differential equations
\begin{subequations}\label{eq:IDFMGeneralizedHodographSystem}
\begin{gather}
W^1_2=W^2_1=\frac12(W^1-W^2),\quad W^1_3=W^2_3=0,\label{eq:IDFMGeneralizedHodographSystem1} \\
W^3_1=W^1-W^3, \quad W^3_2=W^3-W^2. \label{eq:IDFMGeneralizedHodographSystem2}
\end{gather}
\end{subequations}
(Recall that a subscript of a function in $\{1,2,3\}$ denotes
the derivative with respect to the corresponding Riemann invariant.)
Using the equation $W^1_2=W^2_1$, which is in conserved form, and the equations $W^1_3=W^2_3=0$,
we introduce the potential~$\Lambda=\Lambda(\ri^1, \ri^2)$ defined by $\Lambda_1=W^1$ and $\Lambda_2=W^2$
and reduce the subsystem~\eqref{eq:IDFMGeneralizedHodographSystem1} to the single equation
\begin{gather}\label{eq:IDFMHyperbolicKleinGordon}
2\Lambda_{12}=\Lambda_1-\Lambda_2.
\end{gather}
Then we consider the subsystem~\eqref{eq:IDFMGeneralizedHodographSystem2}
as an overdetermined inhomogeneous linear system of first-order partial differential equations
with respect to the function~$W^3$.
The general solution of this system is represented in the form
\[
W^3=F(\ri^3)e^{\ri^2-\ri^1}+\Phi(\ri^1,\ri^2),
\]
where $F$ is an arbitrary smooth function of~$\ri^3$
and $\Phi=\Phi(\ri^1,\ri^2)$ is a particular solution of the subsystem~\eqref{eq:IDFMGeneralizedHodographSystem2},
i.e.,
\begin{gather}\label{eq:IDFMGeneralizedHodographFinalReducedSystem}
\Phi_1+\Phi=\Lambda_1, \quad \Phi_2-\Phi=-\Lambda_2.
\end{gather}
Therefore, the function~$\Phi$ satisfies the same equation as the function~$\Lambda$, $2\Phi_{12}=\Phi_1-\Phi_2$.
It is convenient to assume for~$\Phi$ to be the principal parameter function and to express~$\Lambda$ in terms of~$\Phi$.
Then the determinant $\det(V^i_jt+W^i_j)$, where the indices~$i$ and~$j$ run from~1 to~3, is equal~to
\[
(\Phi_{11}-\Phi_{22})t+2(\Phi_{11}+\Phi_{12})(2\Phi_{12}+\Phi_2)-2(\Phi_{22}+\Phi_{12})(2\Phi_{11}+\Phi_1).
\]
It vanishes on solutions of equation for~$\Phi$ if and only if $\Phi_{11}=\Phi_{22}$ and $\Phi_{11}+\Phi_{12}=0$,
which also gives $\Phi_{22}+\Phi_{12}=0$.
Nevertheless, the equation $\Phi_{11}=\Phi_{22}$ is a differential consequence of
the overdetermined system $2\Phi_{12}=\Phi_1-\Phi_2$, $\Phi_{11}+\Phi_{12}=0$.
The same result holds true if we replace the second equation of this system by $\Phi_{22}+\Phi_{12}=0$.
This is why the nondegeneracy condition $\det(V^i_jt+W^i_j)\ne0$
is equivalent to the single inequality $\Phi_{11}+\Phi_{12}\ne0$ (resp.\ $\Phi_{22}+\Phi_{12}\ne0$).

As a result, we construct the following implicit representation for the general solution of the system~\eqref{eq:IDFMDiagonalizedSystem}:
\begin{gather}\label{eq:IDFMSolutionbyGeneralizedHodograph}
\begin{split}
&x-(\ri^1+\ri^2+1)t=\Phi+\Phi_1,\\
&x-(\ri^1+\ri^2-1)t=\Phi-\Phi_2,\\
&x-(\ri^1+\ri^2)t=\Phi+Fe^{\ri^2-\ri^1},
\end{split}
\end{gather}
where $F$ is an arbitrary smooth function of~$\ri^3$ with $F_{\ri^3}\ne0$
and the function $\Phi=\Phi(\ri^1,\ri^2)$ runs through the set of solutions of the equation $2\Phi_{12}=\Phi_1-\Phi_2$
with $\Phi_{11}\ne-\Phi_{12}$ or, equivalently, $\Phi_{22}\ne-\Phi_{12}$.
The equation for~$\Phi$ is reduced by the transformation
$\tilde\Phi=e^{(\ri^1-\ri^2)/2}\Phi$
to the Klein--Gordon equation $\tilde\Phi_{12}=-\tilde\Phi/4$.

The representation~\eqref{eq:IDFMSolutionbyGeneralizedHodograph} can be derived from
the particular case of the representation~\eqref{eq:IDFMGenSolution},
where the parameter function~$W$ assumed to be nonconstant.
It is just necessary to replace dependent variables in~\eqref{eq:IDFMGenSolution}
by their expressions in terms of the Riemann invariants
and re-denote $\Psi$ and the inverse of~$W$ by~$-\Phi$ and~$-F$, respectively.

Families of solutions from Theorem~\ref{thm:IDFMCompleteSolution} are nicely characterized in terms of conditions for Riemann invariants:
none of (resp.\ exactly one of, resp.\ both) $\ri^1$ and~$\ri^2$ are constants for the regular (resp.\ singular, resp.\ ultra-singular) family.
To complete the set of solutions of the form~\eqref{eq:IDFMSolutionbyGeneralizedHodograph}
to the entire solution set of the system~\eqref{eq:IDFMDiagonalizedSystem},
we should add solutions, where $(\ri^1,\ri^2)$ are defined by the first two equations of~\eqref{eq:IDFMSolutionbyGeneralizedHodograph}
and $\ri^3$ is a constant,
and solutions obtained from the singular and ultra-singular solutions of Theorem~\ref{thm:IDFMCompleteSolution}
by the transition to Riemann invariants.
The corresponding representations in terms of Riemann invariants are obvious.

\section{First-order generalized symmetries and beyond}\label{sec:IDFM1stOrderGeneralizedSymmetries}

To study first-order generalized symmetries%
\footnote{For short, we say that a generalized symmetry is of the first order
if the order of its characteristic as a differential function is not greater than one.
In particular, we consider Lie symmetries as first-order generalized symmetries.
}
of the system~\eqref{eq:IDFModel},
it is more convenient to take advantage of its diagonalized  form~\eqref{eq:IDFMDiagonalizedSystem}
in terms of the Riemann invariants $\ri=(\ri^1,\ri^2,\ri^3)$.
From now on the repeated indices~$i$ and~$j$ mean summation over~$i,j\in\{1,2,3\}$,
while the index~$k$ is fixed and takes the values~$1,2$ or~$3$.
Up to equivalence of generalized symmetries, it suffices to consider only
evolutionary generalized vector fields~\cite{Vinogradov1999,Olver1993} with characteristics of the form
$\eta=(\eta^1(t,x,\ri,\ri_x),\eta^2(t,x,\ri,\ri_x),\eta^3(t,x,\ri,\ri_x))$,
where~$\ri_x$ denotes the first-order $x$-derivative of~$\ri$.

The first-order evolutionary vector field~$Q$ is a generalized symmetry of the system~\eqref{eq:IDFMDiagonalizedSystem}
if it satisfies the condition
\[Q^{(1)}(\ri^k_t+V^k\ri^k_x)|_\eqref{eq:IDFMDiagonalizedSystem}=0,\]
cf.~\eqref{eq:IDFMInfinitesimalInvCriterion} and \cite[Definition~5.2]{Olver1993},
which implies the system $R^k=0$ for the characteristic~$\eta$, where
\begin{gather*}
R^k:=\eta^k_t-\eta^k_{\ri^j}V^j\ri^j_x-\eta^k_{\ri^j_x}\left((\ri^1_x+\ri^2_x)\ri^j_x+V^j\ri^j_{xx}\right)
+(\eta^1+\eta^2)\ri^k_x+V^k(\eta^k_x+\eta^k_{\ri^j}\ri^j_x+\eta^k_{\ri^j_x}\ri^j_{xx}).
\end{gather*}
Splitting the equation~$R^k=0$ with respect to~$\ri^j_{xx}$ readily gives $\eta^k_{\ri^j_x}=0$ for~$j\ne k$, and thus
\[
R^k:=\eta^k_t+V^k\eta^k_x+(V^k-V^j)\eta^k_{\ri^j}\ri^j_x-\eta^k_{\ri^k_x}(\ri^1_x+\ri^2_x)\ri^k_x
+(\eta^1+\eta^2)\ri^k_x.
\]

Then the equations $R^3_{\ri^1_x\ri^1_x}=0$ and $R^3_{\ri^2_x\ri^2_x}=0$
are equivalent to the equations $\eta^1_{\ri^1_x\ri^1_x}=0$ and $\eta^2_{\ri^2_x\ri^2_x}=0$, respectively,
and hence the characteristic components $\eta^1$ and $\eta^2$ can be represented in the form$\mathstrut$
\begin{gather*}
\eta^k=\theta^k(t,x,\ri)\ri^k_x+\zeta^k(t,x,\ri),\quad k=1,2.
\end{gather*}
Using this representation, we further split the entire system $R^k=0$ with respect to $(\ri^1_x,\ri^2_x)$
and, additionally, the equations $R^1=0$ and $R^2=0$ with respect to~$\ri^3_x$.
As a result, we obtain the system
\begin{gather}\label{eq:IDFMqstOrderSymsReducedSystemOfDetEqs}
\begin{split}&
2\theta^1_{\ri^2}=2\theta^2_{\ri^1}=\theta^1-\theta^2,\quad
\zeta^1_{\ri^2}=\zeta^2_{\ri^1}=0,
\\&
\theta^k_{\ri^3}=\zeta^k_{\ri^3}=0,\quad
H^k:=\theta^k_t+V^k\theta^k_x+\zeta^1+\zeta^2=0, \quad
\zeta^k_t+V^k\zeta^k_x=0,\quad k=1,2,
\\&
\ri^3_x\eta^3_{\ri^3_x}+\eta^3_{\ri^1}=\theta^1\ri^3_x,\quad
\ri^3_x\eta^3_{\ri^3_x}-\eta^3_{\ri^2}=\theta^2\ri^3_x,\quad
\eta^3_t+V^3\eta^3_x+(\zeta^1+\zeta^2)\ri^3_x=0.
\end{split}
\end{gather}
In view of $\zeta^1_{\ri^2}=\zeta^2_{\ri^1}=0$,
the equation $\zeta^1_t+V^k\zeta^1_x=0$ (resp.\ $\zeta^2_t+V^k\zeta^2_x=0$)
splits with respect to~$\ri^2$ (resp.~$\ri^1$)
into the equations $\zeta^1_t=0$, $\zeta^1_x=0$ (resp.\ $\zeta^2_t=0$, $\zeta^2_x=0$).
We denote
$K^1:=\ri^3_x\p_{\ri^3_x}+\p_{\ri^1}$,
$K^2:=\ri^3_x\p_{\ri^3_x}-\p_{\ri^2}$,
$B:=\p_t+V^3\p_x$
and derive the following differential consequences of the system~\eqref{eq:IDFMqstOrderSymsReducedSystemOfDetEqs}:
\begin{gather*}
\p_{\ri^2}(\theta^1_t+V^1\theta^1_x+\zeta^1+\zeta^2)=\theta^1_x-\theta^2_x+\zeta^2_{\ri^2}=0,\quad
(K^1B-BK^1)\eta^3=\eta^3_x=(\theta^1_x-\zeta^1_{\ri^1})\ri^3_x,\\
\p_{\ri^1}(\theta^2_t+V^2\theta^2_x+\zeta^1+\zeta^2)=\theta^2_x-\theta^1_x+\zeta^1_{\ri^1}=0,\quad
(BK^2-K^2B)\eta^3=\eta^3_x=(\theta^2_x-\zeta^2_{\ri^2})\ri^3_x.
\end{gather*}
Therefore, $\theta^1_x-\theta^2_x=\zeta^1_{\ri^1}=-\zeta^2_{\ri^2}=\zeta^1_{\ri^1}-\zeta^2_{\ri^2}$,
i.e., $\zeta^1_{\ri^1}=\zeta^2_{\ri^2}=0$, $\theta^1_x=\theta^2_x$, and thus $\eta^3_x=\theta^1_x\ri^3_x$.
We can conclude that $\zeta^1$ and~$\zeta^2$ are constants.
Differentiating the equation $2\theta^1_{\ri^2}=2\theta^2_{\ri^1}=\theta^1-\theta^2$
with respect to~$x$ gives $2\theta^1_{x\ri^2}=2\theta^2_{x\ri^1}=\theta^1_x-\theta^2_x=0$,
and hence $\theta^1_{x\ri^1}=\theta^2_{x\ri^1}=0$ and $\theta^2_{x\ri^2}=\theta^1_{x\ri^2}=0$.
The differential consequence $\p_x(H^1-H^2)=0$ of~\eqref{eq:IDFMqstOrderSymsReducedSystemOfDetEqs}
is equivalent to $\theta^1_{xx}=\theta^2_{xx}=0$.
Then the equations $\p_xH^1=0$ and~$\p_xH^2=0$ reduce to $\theta^1_{tx}=\theta^2_{tx}=0$.
As a result, we have
\begin{gather*}
\theta^k_x=\gamma=\const,\quad
\theta^k_t=-\gamma V^k-(\zeta^1+\zeta^2),\quad
2\theta^1_{\ri^2}=2\theta^2_{\ri^1}=\theta^1-\theta^2,
\\
\eta^3_x=\gamma\ri^3_x,\quad
\eta^3_t=-\gamma V^3\ri^3_x-(\zeta^1+\zeta^2)\ri^3_x,\quad
\ri^3_x\eta^3_{\ri^3_x}+\eta^3_{\ri^1}=\theta^1\ri^3_x,\quad
\ri^3_x\eta^3_{\ri^3_x}-\eta^3_{\ri^2}=\theta^2\ri^3_x.
\end{gather*}
Integrating this system, we derive explicit expressions for~$\theta^1$, $\theta^2$ and~$\eta^3$
and, therefore, for the entire characteristic~$\eta$ of~$Q$,
\noprint{
\begin{gather*}
\theta^1=\gamma x-\big(\gamma V^1+(\zeta^1+\zeta^2)\big)t+\psi,\\
\theta^2=\gamma x-\big(\gamma V^2+(\zeta^1+\zeta^2)\big)t+\psi-2\psi_{\ri^2},
\end{gather*}
where $\psi(\ri^1,\ri^2)$ runs through the set of solutions of the equation $2\psi_{\ri^1\ri^2}=\psi_{\ri^1}-\psi_{\ri^2}$,
and $\gamma$ is a constant.
}
\begin{gather*}
\eta^1=\big(\gamma x-\gamma tV^1-(\zeta^1+\zeta^2)t+\Phi+\Phi_{\ri^1}\big)\ri^1_x+\zeta^1,\\
\eta^2=\big(\gamma x-\gamma tV^2-(\zeta^1+\zeta^2)t+\Phi-\Phi_{\ri^2}\big)\ri^2_x+\zeta^2,\\
\eta^3=\big(\gamma x-\gamma tV^3-(\zeta^1+\zeta^2)t+\Phi\big)\ri^3_x+\Omega,
\end{gather*}
where $\gamma$, $\zeta^1$ and $\zeta^2$ are arbitrary constants,
the function $\Omega=\Omega\big(\ri^3,e^{\ri^2-\ri^1}\ri^3_x\big)$ runs through the set of smooth functions of $\big(\ri^3,e^{\ri^2-\ri^1}\ri^3_x\big)$,
the function $\Phi=\Phi(\ri^1,\ri^2)$ runs through the solution set of the equation $2\Phi_{\ri^1\ri^2}=\Phi_{\ri^1}-\Phi_{\ri^2}$.
This proves the following theorem.

\begin{theorem}\label{thm:IDFMFirstOrderSymmetries}
The algebra~$\Sigma^1$
of first-order reduced generalized symmetries of the system~\eqref{eq:IDFMDiagonalizedSystem} is
spanned by evolutionary vector fields of the form
\begin{gather*}
\check{\mathcal D}=\big(x-(\ri^1+\ri^2+1)t\big)\ri^1_x\p_{\ri^1}+\big(x-(\ri^1+\ri^2-1)t\big)\ri^2_x\p_{\ri^2}+\big(x-(\ri^1+\ri^2)t\big)\ri^3_x\p_{\ri^3},\\
\check{\mathcal G_1}=(t \ri^1_x-1)\p_{\ri^1}+t\ri^2_x\p_{\ri^2}+t\ri^3_x\p_{\ri^3},\quad
\check{\mathcal G_2}=\p_{\ri^1}-\p_{\ri^2},\\
\check{\mathcal P}(\Phi)=(\Phi+\Phi_{\ri^1})\ri^1_x\p_{\ri^1}+(\Phi-\Phi_{\ri^2})\ri^2_x\p_{\ri^2}+\Phi \ri^3_x\p_{\ri^3},\quad
\check{\mathcal W}(\Omega)=\Omega\p_{\ri^3},
\end{gather*}
where the parameter function $\Omega=\Omega\big(\omega^0,\omega^1\big)$
runs through the set of smooth functions of $\omega^0:=\ri^3$ and $\omega^1:=e^{\ri^2-\ri^1}\ri^3_x$,
and the parameter function $\Phi=\Phi(\ri^1,\ri^2)$ runs through
the set of solutions of the equation $2\Phi_{\ri^1\ri^2}=\Phi_{\ri^1}-\Phi_{\ri^2}$,
which is reduced by the substitution $\tilde\Phi=e^{(\ri^1-\ri^2)/2}\Phi$
to the Klein--Gordon equation $\tilde\Phi_{\ri^1\ri^2}=-\tilde\Phi/4$.
\end{theorem}

\begin{remark}\label{rem:IDFMAlgebraOf1stOrderSyms}
The space of first-order reduced generalized symmetries of the system~\eqref{eq:IDFMDiagonalizedSystem} is
closed with respect to the Lie bracket of generalized vector fields,
and hence we can call it an algebra.
This property is shared by all strictly hyperbolic diagonalizable hydrodynamic-type systems. 
Up to skew-symmetry of Lie bracket, nonzero commutation relations among the above vector fields
are exhausted by the following ones:
\begin{gather*}
[\check{\mathcal D},\check{\mathcal P}(\Phi)]=\check{\mathcal P}(\Phi),\quad
[\check{\mathcal G_1},\check{\mathcal P}(\Phi)]=\check{\mathcal P}(-\Phi_{\ri^1}),\quad
[\check{\mathcal G_2},\check{\mathcal P}(\Phi)]=\check{\mathcal P}(\Phi_{\ri^1}-\Phi_{\ri^2}),\quad
\\
[\check{\mathcal D},\check{\mathcal W}(\Omega)]=\check{\mathcal W}(\omega^1\Omega_{\omega^1}),\quad
[\check{\mathcal G_1},\check{\mathcal W}(\Omega)]=\check{\mathcal W}(\omega^1\Omega_{\omega^1}),\quad
[\check{\mathcal G_2},\check{\mathcal W}(\Omega)]=-2\check{\mathcal W}(\omega^1\Omega_{\omega^1}),
\\
[\check{\mathcal W}(\Omega^1),\check{\mathcal W}(\Omega^2)]=\check{\mathcal W}\big(
\Omega^1_{\omega^0}(\omega^1\Omega^2_{\omega^1}-\Omega^2)-
\Omega^2_{\omega^0}(\omega^1\Omega^1_{\omega^1}-\Omega^1)
\big).
\end{gather*}
Therefore, the subspaces~$\mathcal I^1$ and~$\mathcal I^2$
that consist of all generalized vector fields of the forms
$\check{\mathcal P}(\Phi)$ and $\check{\mathcal W}(\Omega)$
from the algebra~$\Sigma^1$, respectively,
are (infinite-dimensional) ideals of~$\Sigma^1$.
Moreover, the ideal~$\mathcal I^1$ is commutative.
Since $\check{\mathcal P}(e^{\ri^2-\ri^1})=\check{\mathcal W}(\omega^1)=e^{\ri^2-\ri^1}\ri^3_x\p_{\ri^3}$,
the above ideals are not disjoint,
$\mathcal I^1\cap\mathcal I^2=\langle e^{\ri^2-\ri^1}\ri^3_x\p_{\ri^3}\rangle$.
\end{remark}

\begin{remark}\label{rem:IDFMGeneralizedSymmetries}
Amongst the generalized vector fields presented in Theorem~\ref{thm:IDFMFirstOrderSymmetries},
there are evolutionary forms of Lie symmetries of the system~\eqref{eq:IDFMDiagonalizedSystem}
(cf. Remark~\ref{rem:IDFMDiagonalization}) as well as genuinely generalized symmetries.
The generalized vector fields~$\check{\mathcal D}$, $\check{\mathcal G_1}$, $\check{\mathcal G_2}$,
$\check{\mathcal P}(\ri^1{+}\ri^2)$, $\check{\mathcal P}(1)$ and $\check{\mathcal W}(\Omega)$
with $\Omega$ depending on~$\ri^3$ alone are evolutionary forms of the vector fields
$-\hat{\mathcal D}$, \smash{$-\frac12\hat{\mathcal G}-\frac12\hat{\mathcal P}^v$},
$\hat{\mathcal P}^v$, $\hat{\mathcal P}^t$, $-\hat{\mathcal P}^x$ and~$\hat{\mathcal W}(\Omega)$, respectively.
The vector fields~$\check{\mathcal W}(\Omega)$
with $\Omega$ not being a function of~$\ri^3$ only are genuine generalized symmetries.
The existence of generalized symmetries~$\check{\mathcal P}(\Phi)$ with~$\Phi\notin\mathop{\rm span}(1,\ri^1{+}\ri^2)$
can be explained by observing that the pushforward of the vector field~$\check{\mathcal P}(\Phi)$
by the projection to $(t,x,\ri^1,\ri^2)$
is the evolutionary form of the Lie symmetry vector field~$\breve{\mathcal P}(\tau^0,\xi^0)$ of the essential subsystem~$\mathcal S_0$
(re-written in terms of~$(\ri^1,\ri^2)$)
with $\tau^0=\frac12(\Phi_{\ri^1}+\Phi_{\ri^2})$ and $\xi^0=\frac12(\Phi_{\ri^1}+\Phi_{\ri^2})(\ri^1+\ri^2-1)-\Phi+\Phi_{\ri^2}$.
Therefore, Lie symmetries of the essential subsystem that are lost upon transition to the entire system
are recovered as generalized first-order symmetries of the entire system.
In contrast to the family $\{\breve{\mathcal P}(\tau^0,\xi^0)\}$, the other distinguished Lie symmetry of the essential subsystem~$\mathcal S_0$,
$\breve{\mathcal J}$, is related to $\zeta^1_{\ri^1}$, which vanishes in the proof of Theorem~\ref{thm:IDFMFirstOrderSymmetries},
and therefore this symmetry has no counterpart amongst first-order symmetries of the system~$\mathcal S$.
\end{remark}

First-order symmetries from the subspace $\{\check{\mathcal W}(\Omega)\}$
can be directly generalized to an arbitrary order.

\begin{proposition}\label{pro:IDFMDiagonalizedSystemGenSymsWithWk}
The system~\eqref{eq:IDFMDiagonalizedSystem} admits generalized symmetries of arbitrarily high order of the form
$\check{\mathcal W}(\Omega)=\Omega\p_{\ri^3}$,
where $\Omega$ runs through the set of smooth functions
of a finite, but unspecified number of $\omega^\iota=(e^{\ri^2-\ri^1}\mathrm D_x)^\iota\ri^3$, $\iota\in\mathbb N_0$,
i.e., $\Omega=\Omega(\omega^0,\dots,\omega^\kappa)$ with $\kappa\in \mathbb N_0$.
\end{proposition}

\begin{proof}
The differential operator~$\mathrm A=e^{\ri^2-\ri^1}\mathrm D_x$ commutes with the operator~$\mathrm B=\mathrm D_t+(\ri^1+\ri^2)\mathrm D_x$
on solutions of the system~\eqref{eq:IDFMDiagonalizedSystem}
since \[[\mathrm B,\mathrm A]=(\mathrm B\ri^2-\ri^2_x-\mathrm B\ri^1-\ri^1_x)\mathrm A.\]
Therefore, the operator~$\mathrm A$ maps solutions of the equation~\eqref{eq:IDFMDEquation3}
to solutions of the same equation.
In particular, the equations $\mathrm B\omega^\kappa=0$ are differential consequences of the system~\eqref{eq:IDFMDiagonalizedSystem}.
This observation as well as the fact that the equation~\eqref{eq:IDFMDEquation3}
is both linear and of first order with respect to~$\ri^3$
hints that $\check{\mathcal W}(\Omega)$ is a generalized symmetry of the system~\eqref{eq:IDFMDiagonalizedSystem}.
Indeed, it satisfies the invariance criterion.
\end{proof}

In fact, the system~\eqref{eq:IDFMDiagonalizedSystem} admits
another family of generalized symmetries of arbitrarily high order,
which essentially differs from the family $\{\check{\mathcal W}(\Omega)\}$
and is obtained by the prolongation of generalized symmetries
of the essential subsystem to~$\ri^3$.
In the next paper on the system~\eqref{eq:IDFMDiagonalizedSystem},
we shall exhaustively describe all generalized symmetries of this system.

\section{Hydrodynamic conservation laws and their generalizations}\label{sec:IDFMHydrodynamicCLs}

Here we find the complete space of zeroth-order conservation laws of the system~\eqref{eq:IDFMDiagonalizedSystem},
which is in fact exhausted by linear combinations of a single non-translation-invariant conservation law with hydrodynamic conservation laws.
Recall that a conservation law is called \textit{hydrodynamic} if its density~$\rho$ is a function of dependent variables only.
Following~\cite{Doyle1994}, we also present the complete set of first-order conservation laws
whose densities do not involve independent variables and a family of conservation laws of arbitrarily high order.

\begin{theorem}\label{thm:IDFMConservationLaws}
The space of zeroth-order conservation laws of the system~\eqref{eq:IDFMDiagonalizedSystem}
is spanned by the single non-translation-invariant conservation law with the conserved current
\[e^{\ri^1-\ri^2}\big(x-V^3t,V^3(x-V^3t)-t\big)\]
and the space of hydrodynamic conservation laws of this system,
which is naturally isomorphic to the space of conserved currents of the general form
\[
\big(e^{\ri^1-\ri^2}\Omega+\Psi_{\ri^1}-\Psi_{\ri^2},\,(\ri^1+\ri^2)e^{\ri^1-\ri^2}\Omega+V^1\Psi_{\ri^1}-V^2\Psi_{\ri^2}\big).
\]
Here~$\Omega$ is an arbitrary smooth function of~$\ri^3$
and the parameter function~$\Psi=\Psi(\ri^1,\ri^2)$ runs through the solution set of the equation
$2\Psi_{\ri^1\ri^2}=\Psi_{\ri^2}-\Psi_{\ri^1}$, which is reduced to the Klein--Gordon equation $\tilde \Psi_{\ri^1\ri^2}=-\tilde \Psi/4$
by the transformation $\tilde\Psi=e^{(\ri^2-\ri^1)/2}\Psi$.
\end{theorem}

\begin{proof}
If~$\rho=\rho(t,x,\ri)$ is a density of a conservation law of the system~\eqref{eq:IDFMDiagonalizedSystem},
then $\bar{\mathsf E}\bar{\mathrm D}_t\rho=0$,
where \smash{$\bar{\mathrm D}_t=\p_t-\sum_{i=1}^3V^i\ri^i_x\p_{\ri^i}$}
is the restricted reduced operator of total derivative with respect to~$t$ in view of the system~\eqref{eq:IDFMDiagonalizedSystem},
$\bar{\mathsf E}=\big(\p_{\ri^i}-\mathrm D_x\p_{\ri^i_x},\,j=1,2,3\big)^{\mathsf T}$
is the restricted Euler operator.
The above condition is equivalent to the system
\[
\sum_{i=1}^3\big((V^j-V^i)\rho_{\ri^i\ri^j}+V^j_{\ri^i}\rho_{\ri^j}-V^i_{\ri^j}\rho_{\ri^i}\big)\ri^i_x
+\rho_{\ri^jt}+V^j\rho_{\ri^jx}=0,\quad j=1,2,3.
\]
Splitting this system with respect to~$\ri^i_x$ gives three equations corresponding to summands without derivatives of~$\ri$
and a system associated with coefficients of $\ri^i_x$, $i=1,2,3$,
whose left-hand side is antisymmetric with respect to the permutation of~$i$ and~$j$.
Hence the latter system contains only three independent differential equations on~$\rho$.
As a result, we derive a system of six determining equations for~$\rho$,
\begin{gather}\label{eq:IDFMDDetEqsFor0th-orderCLs}
\rho_{\ri^1\ri^3}=\rho_{\ri^3},\quad
\rho_{\ri^2\ri^3}=-\rho_{\ri^3},\quad
2\rho_{\ri^1\ri^2}=\rho_{\ri^2}-\rho_{\ri^1},\quad
\rho_{t\ri^j}+V^j\rho_{x\ri^j}=0,\quad j=1,2,3.
\end{gather}
Denoting $R^0:=\rho_{\ri^1\ri^3}-\rho_{\ri^3}$ and $R^j:=\rho_{t\ri^j}+V^j\rho_{x\ri^j}$, $j=1,2,3$,
We derive the differential consequence
$\p_{\ri^3}R^1-(\p_t+V^1\p_x)R^0-R^3=\rho_{\ri^3x}=0$,
and thus also $\rho_{\ri^3t}=0$.
The first two equations of the system~\eqref{eq:IDFMDDetEqsFor0th-orderCLs}
can be rewritten as $(e^{\ri^2-\ri^1}\rho_{\ri^3})_{\ri^1}=(e^{\ri^2-\ri^1}\rho_{\ri^3})_{\ri^2}=0$.
Therefore, the general solution of~\eqref{eq:IDFMDDetEqsFor0th-orderCLs} admits the representation
\[\rho=e^{\ri^1-\ri^2}\Omega+\tilde\rho(t,x,\ri^1,\ri^2),\]
where $\Omega=\Omega(\ri^3)$ is an arbitrary smooth function of~$\ri^3$,
and $\tilde\rho=\tilde\rho(t,x,\ri^1,\ri^2)$ is the general solution of the reduced system
\begin{gather}\label{eq:IDFMDReducedDetEqsFor0th-orderCLs}
2\tilde\rho_{\ri^1\ri^2}=\tilde\rho_{\ri^2}-\tilde\rho_{\ri^1},\quad
\tilde\rho_{t\ri^j}+V^j\tilde\rho_{x\ri^j}=0,\quad j=1,2.
\end{gather}
The last two equations of~\eqref{eq:IDFMDReducedDetEqsFor0th-orderCLs} integrate to
$\tilde\rho_{\ri^j}=f^j(y_j,\ri^1,\ri^2)$, $j=1,2$,
where $f^1$ and~$f^2$ are smooth functions of their arguments, and $y_j=x-V^jt$.
The consistency of these representations for the derivatives~$\tilde\rho_{\ri^1}$ and~$\tilde\rho_{\ri^2}$
with each other and with the first equations of~\eqref{eq:IDFMDReducedDetEqsFor0th-orderCLs}
leads to equations for~$f^1$ and~$f^2$,
\begin{gather}\label{eq:IDFMDDoubleReducedDetEqsFor0th-orderCLs}
 (y_1-y_2)f^1_{y_1}+2f^1_{\ri^2}
=(y_1-y_2)f^2_{y_2}+2f^2_{\ri^1}
=f^2-f^1.
\end{gather}
Successively differentiating these equations with respect to~$y_1$ and~$y_2$,
we get $f^1_{y_1y_1}=f^2_{y_2y_2}=0$,
i.e., $f^j=g^j(\ri^1,\ri^2)y_j+h^j(\ri^1,\ri^2)$
for some smooth functions~$g^j$ and~$h^j$ of $(\ri^1,\ri^2)$, $j=1,2$.
Then the splitting of the equations~\eqref{eq:IDFMDDoubleReducedDetEqsFor0th-orderCLs}
with respect to~$y_1$ and~$y_2$ gives a system
$g^1_{\ri^2}=-g^1$, $g^2_{\ri^1}=g^2$, $g^1=-g^2$, \smash{$2h^1_{\ri^2}=2h^2_{\ri^1}=h^2-h^1$}
for the coefficients~$g^j$ and~$h^j$.
The general solution of this system can be represented in the form
$g^1=-g^2=Ce^{\ri^1-\ri^2}$, $h^1=\Psi_{\ri^1\ri^2}-\Psi_{\ri^1\ri^1}$, $h^2=\Psi_{\ri^2\ri^2}-\Psi_{\ri^1\ri^2}$,
where $C$ is an arbitrary constant
and $\Psi=\Psi(\ri^1,\ri^2)$ is an arbitrary solution of the equation $2\Psi_{\ri^1\ri^2}=\Psi_{\ri^2}-\Psi_{\ri^1}$.
Simultaneously integrating the equations $\tilde\rho_{\ri^j}=f^j$ in view of the derived expressions for~$f_j$, $j=1,2$,
we finally obtain
\[\tilde\rho=Ce^{\ri^1-\ri^2}(x-V^3t)+\Psi_{\ri^1}-\Psi_{\ri^2}+\rho^0\]
with $\rho^0$ being an arbitrary smooth function of~$(t,x)$,
which is the density of a null divergence and should thus be neglected.
To find the flux~$\sigma$ associated with~$\rho=e^{\ri^1-\ri^2}\Omega+\tilde\rho$, we notice that
$
\mathrm D_x\sigma=-\bar{\mathrm D}_t\rho=\mathrm D_x\big((\ri^1+\ri^2)e^{\ri^1-\ri^2}\Omega+C(V^3(x-V^3t)-t)+V^1\Psi_{\ri^1}-V^2\Psi_{\ri^2}\big).
$
\end{proof}

The conservation-law characteristic associated with the above conserved current is
\[
(e^{\ri^1-\ri^2}\Omega+\Psi_{\ri^1\ri^1}-\Psi_{\ri^1\ri^2},\,
-e^{\ri^1-\ri^2}\Omega+\Psi_{\ri^1\ri^2}-\Psi_{\ri^2\ri^2},\,
e^{\ri^1-\ri^2}\Omega_{\ri^3}).
\]

In addition to the standard equivalence of conservation laws one can consider their equivalence
up to the action of a point symmetry group of the system of differential equations under study,
see~\cite{Khamitova1982,PopovychKunzingerIvanova2008} for details.
A set of conservation laws that generate a space of conservation laws
via linear combining and acting with point symmetries is called a \emph{generating set}
for this space~\cite{Khamitova1982}.

There are only two inequivalent values of the parameter function~$\Omega(\ri^3)$
up to transformations in the point symmetry group~$G$ of the system~\eqref{eq:IDFMDiagonalizedSystem}
(see Theorem~\ref{thm:IDFMSymmetryGroup}),
$\Omega=1$ and $\Omega=\ri^3$.
Moreover, it is easily seen that
the conserved current with $(\Omega,\Psi)=(1,0)$ coincides with
the conserved current with $(\Omega,\Psi)=(0,e^{\ri^1-\ri^2})$.

\begin{corollary}\label{eq:IDFMDiagonalizedSystemHydrodynamicCLs}
A generating set for the space of hydrodynamic conservation laws of the system~\eqref{eq:IDFMDiagonalizedSystem}
up to the point symmetry group~$G$ of this system
consists of the conservation laws with conserved currents
\begin{gather*}
\mathrm{DHC}=\left(e^{\ri^1-\ri^2}\ri^3,\ (\ri^1+\ri^2)e^{\ri^1-\ri^2}\ri^3\right),\\
\mathrm{EHC}(\Psi)=\big(\Psi_{\ri^1}-\Psi_{\ri^2},\ (\ri^1+\ri^2+1)\Psi_{\ri^1}-(\ri^1+\ri^2-1)\Psi_{\ri^2}\big),
\end{gather*}
where~$\Psi=\Psi(\ri^1,\ri^2)$ runs through the solution set of the equation $2\Psi_{\ri^1\ri^2}=\Psi_{\ri^2}-\Psi_{\ri^1}$,
which is reduced by the transformation $\tilde\Psi=e^{(\ri^2-\ri^1)/2}\Psi$
to the Klein--Gordon equation $\tilde\Psi_{\ri^1\ri^2}=-\tilde\Psi/4$.
\end{corollary}

\begin{remark}
Functions~$\Phi$ and~$\Psi$ parameterizing first-order generalized symmetries
and hydrodynamic conservation laws of the system~\eqref{eq:IDFMDiagonalizedSystem}
satisfy adjoint differential equations. This can be explained by the fact that
characteristics of conservations laws of the system~\eqref{eq:IDFMDiagonalizedSystem}
are cosymmetries of this system, cf.~\cite{Vinogradov1984}.
\end{remark}

For the system~\eqref{eq:IDFMModelForGroupAnalysis},
the space of hydrodynamic conservation laws
is naturally isomorphic to the space of conserved currents of the general form
$(e^v\Omega+\Psi_v,\,ue^v\Omega+u\Psi_v+\Psi_u)$,
where~$\Omega=\Omega(w)$ is an arbitrary smooth function of~$w$
and the parameter function~$\Psi=\Psi(u,v)$ runs through the solution set of the telegraph equation
$\Psi_{uu}=\Psi_{vv}-\Psi_v$.
The conservation-law characteristic associated with the above conserved current is
$(\Psi_{uv},e^v\Omega+\Psi_{vv},e^v\Omega_w)$.
A generating set of the space of hydrodynamic conservation laws of the system~$\mathcal S$
consists of the conservation laws with conserved currents
$(e^vw,ue^vw)$, $(\Psi_v,u\Psi_v+\Psi_u)$,
where the parameter function~$\Psi=\Psi(u,v)$ runs through the solution set of the telegraph equation
$\Psi_{uu}=\Psi_{vv}-\Psi_v$.

Many semi-Hamiltonian systems, including the system~\eqref{eq:IDFMDiagonalizedSystem},
also possess first-order conservation laws~\cite{Doyle1994,Sheftel1994b}
(see also~\cite{Serre1989,Serre1991}).
To construct such conservation laws for the system~\eqref{eq:IDFMDiagonalizedSystem} following~\cite[Theorem~5.1]{Doyle1994},
we should first find some smooth nonvanishing functions~$G^i(\ri)$, $i=1,2,3$, that satisfy the equations
\[
\frac{G^i_{\ri^j}}{G^i}=-\frac{V^i_{\ri^j}}{V^i-V^j},\quad i\neq j.
\]
We can take $G^1=e^{-\ri^2/2}$, $G^2=e^{\ri^1/2}$, $G^3=e^{\ri^1-\ri^2}$.
Every first-order conserved current of the system~\eqref{eq:IDFMDiagonalizedSystem}
with density not involving $(t,x)$
is equivalent to the sum of a hydrodynamic conserved current and a conserved current of the form
\[
\left(\sum\limits_{i=1}^2\frac{(G^i)^2f^i}{\ri^i_x}+G^3\Omega,\,
\sum\limits_{i=1}^2V^i\frac{(G^i)^2f^i}{\ri^i_x}+V^3G^3\Omega-2x\sum\limits_{i=1}^2V^i_{\ri^i}\frac{(G^i)^2f^i}{\ri^i_x}\right),
\]
where~$\Omega=\Omega(\ri^3,\ri^3_x/G^3)$ is an arbitrary smooth function of its arguments,
and $f^1=f^1(\ri^1)$ and $f^2=f^2(\ri^2)$ are arbitrary smooth functions of~$\ri^1$ and~$\ri^2$,
respectively, that satisfy the condition
\[
\sum\limits_{i=1}^2(G^i)^2f^i=\const.
\]
(The above summations are in the range $\{1,2\}$ since $V^3_{\ri^3}=0$.)
Therefore, $f^1(\ri^1)=ce^{\ri^1}$ and $f^2(\ri^2)=-ce^{-\ri^2}$,
where $c$ is an arbitrary constant.

\begin{proposition}\label{pro:IDFMFirstOrderCurrent}
The space of conserved currents of conservation laws of the system~\eqref{eq:IDFMDiagonalizedSystem} of order not greater than one
with $(t,x)$-independent densities is spanned by a family of conserved currents $\mathrm C_1(\Omega)$, $\mathrm C_0$ and the
conserved currents $\mathrm{EHC}(\Psi)$ presented in Corollary~\ref{eq:IDFMDiagonalizedSystemHydrodynamicCLs}. Here
\begin{gather}\label{eq:IDFMFirstOrderCurrent}
\begin{split}
&\mathrm C_0=
\left(\left(\frac1{\ri^1_x}-\frac1{\ri^2_x}\right)e^{\ri^1-\ri^2},
\left(\frac{V^1}{\ri^1_x}-\frac{V^2}{\ri^2_x}\right)e^{\ri^1-\ri^2}\right),
\\[1ex]
&\mathrm C_1(\Omega)=\left(e^{\ri^1-\ri^2}\Omega, (\ri^1+\ri^2)e^{\ri^1-\ri^2}\Omega\right),
\end{split}
\end{gather}
and~$\Omega$ runs through the set of smooth functions of $\omega^0=\ri^3$ and $\omega^1=\ri^3_xe^{\ri^2-\ri^1}$.
\end{proposition}

Since the system~$\mathcal S$ is not genuinely nonlinear, Theorem~5.2 of~\cite{Doyle1994}
implies the following assertion.

\begin{proposition}
The system~$\mathcal S$ possesses a family of nontrivial conservation laws of arbitrarily high order
with conserved currents~$\mathrm C_1(\Omega)$
parameterized by an arbitrary function~$\Omega$ of
a finite, but unspecified number of $\omega^\iota=(e^{\ri^2-\ri^1}\mathrm D_x)^\iota\ri^3$, $\iota\in\mathbb N_0$.
\end{proposition}

\begin{remark}\label{rem:IDFMFamilyOfCLsWithOmegas}
In a fashion different from~\cite{Doyle1994},
existence of higher-order conservation laws for the system~\eqref{eq:IDFMDiagonalizedSystem}
can be explained by actions of generalized symmetries $\check{\mathcal W}(\Omega)=\Omega\p_{\ri^3}$
on a hydrodynamic conserved current~$\mathrm{DHC}$,
see Proposition~\ref{pro:IDFMDiagonalizedSystemGenSymsWithWk}
and Corollary~\ref{eq:IDFMDiagonalizedSystemHydrodynamicCLs}.
This claim agrees with Proposition~\ref{pro:IDFMFirstOrderCurrent}.
\end{remark}

In fact, the system~\eqref{eq:IDFMDiagonalizedSystem} admits
another family of conservation laws of arbitrarily high order,
which essentially differs from the family of conserved currents $\mathrm C_1(\Omega)$
and is obtained by pulling back conservation laws of the essential subsystem
with the projection to~$(t,x,\ri^1,\ri^2)$.
In the next paper on the system~\eqref{eq:IDFMDiagonalizedSystem},
we shall exhaustively describe all local conservation laws of this system,
including $(t,x)$-dependent ones.

\section{Conclusions}\label{sec:IDFMConclusions}

In the present paper we have performed an extended symmetry analysis of the isothermal no-slip drift
flux model given by the system~\eqref{eq:IDFModel}. It turned out that the form~\eqref{eq:IDFModel} was not
convenient for the study within the symmetry framework. In particular, the maximal Lie invariance algebra
of the system~\eqref{eq:IDFModel} is difficult to compute even
using specialized computer algebra packages, e.g. DESOLVII~\cite{VuJeffersonCarminati2012}.
Therefore, transforming dependent variables we have represented the model in the form~\eqref{eq:IDFMModelForGroupAnalysis}.
This representation has allowed us to compute the maximal Lie invariance algebra~$\mathfrak g$
of the initial system. This algebra is infinite-dimensional. Note that the infinite-dimensional part of~$\mathfrak g$
spanned by~$\{\tilde{\mathcal W}(\tilde\Omega)\}$ was missed in~\cite{SekharSatapathy2016}.

Moreover, we have computed the complete point symmetry group of the system~$\mathcal S$
using the combined algebraic method. Since the algebra~$\mathfrak g$ is infinite-dimensional,
the straightforward application of the automorphism-based method is not appropriate.
For this reason we have chosen to employ its megaideal-based counterpart.
However, the fact that the finite-dimensional part of the algebra~$\mathfrak g$ coincides with its radical~$\mathfrak r$,
which is a megaideal, allowed us to partially use the automorphism-based version of the algebraic method as well.
To this end we have computed the entire automorphism group of the
radical~$\mathfrak r$ of~$\mathfrak g$.
Finally, we have employed the constraints derived from the algebraic method
to obtain the complete point symmetry group~\eqref{eq:IDFMSymmetryGroup}
of the isothermal no-slip drift flux model~$\mathcal S$ applying the direct method.

Following the standard Lie reduction procedure~\cite{Ovsiannikov1982}, we have also obtained
optimal lists of one- and two-dimensional subalgebras of the algebra~$\mathfrak g$,
which were employed for finding appropriate solution ansatzes, using which we have
found families of invariant and  partially invariant solutions of the system~$\mathcal S$.

Sections~\ref{sec:IDFMEssentialSubsystem}--\ref{sec:IDFMSolutionsViaEssentialSubsystem}
and~\ref{sec:IDFMGeneralizedHodographMethod} of the present paper provide
two alternative approaches to finding solutions of the system under study. The first of them
is based on the fact that the subsystem~$\mathcal S_0$ is partially decoupled from the rest of the system,
so we could construct its solutions first and then solve the remaining equation~\eqref{eq:IDFMEquation3} in
parametric form. This choice is justified by the fact that the system~$\mathcal S_0$ admits a larger Lie symmetry group
than the Lie symmetry group of~$\mathcal S$, and thus one can find more invariant or partially invariant
solutions. Moreover, one should keep in mind the well-known fact that every hydrodynamic-type system
in two independent variables, say, $x,t$, and two dependent variables, say, $u,v$,
can be linearized via the two-dimensional hodograph transformation
with~$(p,q)=(t,x)$, $(y,z)=(u,v)$ being the new dependent and independent variables, respectively.
It is worth noticing that upon having applied the above hodograph transformation to $\mathcal S_0$
the equation~\eqref{eq:IDFMEquation3} in the new variables remains
a first-order linear PDE in~$w$ which can be easily solved.
The above change of variables has led to the potential system~\eqref{eq:IDFMPotentialTelegraphEquation}
of the telegraph equation~\eqref{eq:IDFMTelegraphEquation}.
Using the subsequent change of variables $\tilde p=p\exp(-z/2)$ we obtained
the famous Klein--Gordon equation~\eqref{eq:IDFMKleinGordonEquation}.
As the latter equation is well studied, we could construct
large classes of exact solutions thereof, find~$p$ and solve equation~\eqref{eq:IDFMEquation3},
which allows us to find solutions of the initial system $\mathcal S$ in parametric form.

The second approach to finding solutions for~\eqref{eq:IDFMModelForGroupAnalysis} is
based on the semi-Hamil\-tonian property of~\eqref{eq:IDFMDiagonalizedSystem},
which has allowed us to employ the generalized hodograph transformation.
Just as for the previous method, the problem of finding solutions of the system~\eqref{eq:IDFMModelForGroupAnalysis}
was reduced to solving the Klein--Gordon equation.
Nevertheless, the solutions obtained via generalized hodograph transformations
are included in the list of solutions obtained via the decoupling of~\eqref{eq:IDFMModelForGroupAnalysis}.

We have computed the first-order generalized symmetries for the system~\eqref{eq:IDFMDiagonalizedSystem},
amongst which we found both evolutionary forms of point symmetries and genuinely generalized symmetries.
Interestingly, some of genuinely generalized symmetries are nothing but
evolutionary forms of point symmetries of the essential subsystem. Moreover, the system~$\mathcal S$ possesses
genuinely generalized symmetries beyond the latter, see Remark~\ref{rem:IDFMGeneralizedSymmetries}.

We have also constructed the entire space of zero-order conservation laws for the system~\eqref{eq:IDFMModelForGroupAnalysis},
which is spanned by a single non-translation-invariant conservation law
and the (infinite-dimen\-sional) space of hydrodynamic conservation laws of this system.
Note that while the existence of hydrodynamic-type systems with an infinite number of independent hydrodynamic conservation laws
was known for a long time~\cite{Nutku1987,OlverNutku1988},
hydrodynamic-type systems with non-hydrodynamic conservation laws are significantly less common,
cf.\ e.g.\ \cite{FerapontovSharipov1996} and references therein.
The space of hydrodynamic conservation laws of a diagonalizable semi-Hamiltonian system is parameterized
by~$n$ arbitrary functions of one variable~\cite{Tsarev1991}, where~$n$ is the number of dependent variables.

For the system~\eqref{eq:IDFMModelForGroupAnalysis} two of the arbitrary functions in question are determined by solutions
of the telegraph equation $\Phi_{uu}=\Phi_{vv}-\Phi_v$,
while the third function is an arbitrary function of~$w$, which can be reduced to~$1$ by actions of point symmetries
unless it is zero.
The first-order conservation laws with $(t,x)$-independent densities have also been classified.
Existence of higher-order conservation laws has been justified using the fact that system~\eqref{eq:IDFMDiagonalizedSystem}
is not genuinely nonlinear.
Such conservation laws have been found using the action of generalized symmetries
of the system~\eqref{eq:IDFMDiagonalizedSystem} of arbitrarily high order.

We found only generalized symmetries and conservation laws of the system~$\mathcal S$ of order not greater than one
as well as higher-order generalized symmetries and conservation laws related to the linearly degenerate part of the system~$\mathcal S$.

In general, it is quite a common situation when a hydrodynamic-type system does not possess
generalized symmetries of order greater than one.
The key observation facilitating the study of higher-order symmetries for the system~$\mathcal S$
is the fact that this system is only partially coupled,
and the essential subsystem~$\mathcal S_0$ of the system~$\mathcal S$ is linearized
by the hodograph transformation to the Klein--Gordon equation
for which all generalized symmetries are known~\cite{NikitinPrylypko1990a,OpanasenkoPopovych2018,ShapovalovShirokov1992}.
The main conjecture here is that the entire space of generalized symmetries of the system~$\mathcal S$ is spanned
by generalized symmetries from the family presented in Proposition~\ref{pro:IDFMDiagonalizedSystemGenSymsWithWk}
and prolongations of generalized symmetries of the essential subsystem~$\mathcal S_0$ to~$\ri^3$.
The further work is to figure out which of the generalized symmetries of~$\mathcal S_0$
are locally prolonged to generalized symmetries of the entire system~$\mathcal S$
and how this prolongation is performed.

Linearizability of the subsystem~$\mathcal S_0$
can also be applied to classification of all (local) conservation laws of the system~$\mathcal S$.
Our conjecture here is that the entire space of conservation laws of the system~$\mathcal S$ is spanned by
conservation laws from the family presented in Remark~\ref{rem:IDFMFamilyOfCLsWithOmegas}
and by those being pullbacks of conservation laws the essential subsystem~$\mathcal S_0$
by the projection to $(t,x,\ri^1,\ri^2)$.
Moreover, one does not need to prolong conservation laws of the essential subsystem~$\mathcal S_0$
to conservation laws of the entire system~$\mathcal S$,
which makes the problem easier than that of finding generalized symmetries.

It is readily verified that system~\eqref{eq:IDFMDiagonalizedSystem} admits
a one-parameter family of Hamiltonian operators of Dubrovin--Novikov type
(see e.g.\ the surveys \cite{DubrovinNovikov1989, Ferapontov1995} for more details on such operators)
with the parameter $\lambda$,
\begin{gather}
\begin{split}\label{eq:IDFMHamiltonianOperator}
\mathcal{P}_\lambda={}&e^{\ri^2-\ri^1}\mathop{\rm diag}\big(1,-1,\lambda e^{\ri^2-\ri^1}\big)\mathrm D_x
\\[1ex]
&+\frac12e^{\ri^2-\ri^1}
\left(\begin{array}{ccc}
\ri^2_x-\ri^1_x  &  \ri^1_x-\ri^2_x  &  -2\ri^3_x\\[.5ex]
\ri^2_x-\ri^1_x  &  \ri^1_x-\ri^2_x  &  -2\ri^3_x\\[.5ex]
2\ri^3_x         &  2\ri^3_x         &  -2\lambda e^{\ri^2-\ri^1}\left(\ri^1_x-\ri^2_x\right)
\end{array}\right),
\end{split}
\end{gather}
and the associated Hamiltonian representation
\[
\boldsymbol{\mathfrak{r}}_t=\mathcal{P}_\lambda (\delta H/\delta\boldsymbol{\ri}),
\]
where $\boldsymbol{\mathfrak{r}}=(\ri_1,\ri_2,\ri_3)^{\mathrm{T}}$, the Hamiltonian functional is given by
\[
H=-\frac14\int e^{\ri_1-\ri_2}\left((\ri_1+\ri_2)^2+2(\ri_1-\ri_2)\right)\,{\rm d}x,
\]
$\delta/\delta\boldsymbol{\mathfrak{r}}$ is the variational derivative,
and the integral is understood in the formal sense of calculus of variations;
cf.\ e.g.\ \cite[Chapter~4]{Olver1993} for the details on the latter.

The Hamiltonian operators \eqref{eq:IDFMHamiltonianOperator} can be formally inverted if $\lambda\neq 0$,
and in view of this and of the above the ratio $\mathcal{R}_{\mu\nu}=\mathcal{P}_\mu \circ\mathcal{P}_\nu^{-1}$ defines,
for $\nu\neq 0$, a hereditary recursion operator for~\eqref{eq:IDFMDiagonalizedSystem};
see e.g.\ \cite{Olver1993, AS17} and references therein for more details on recursion operators in general.
Alas, the leading term of $\mathcal{R}_{\mu\nu}-\mathbb{I}$, where $\mathbb{I}$ is the identity operator,
is a degenerate matrix,
which makes $\mathcal{R}_{\mu\nu}$ rather uninteresting
from the point of view of generating symmetries for~\eqref{eq:IDFMDiagonalizedSystem}.

In closing note that \eqref{eq:IDFMDiagonalizedSystem} also admits infinitely many Hamiltonian operators
that are not of Dubrovin--Novikov type, and a recursion operator which is free of the degeneracy of the above kind.
Also, it is well known that the Hamiltonian operators map cosymmetries (in particular, characteristics of conservation laws)
into symmetries, which could lead to an interesting interplay among the two for the system~\eqref{eq:IDFMDiagonalizedSystem},
and hence for $\mathcal{S}$.
All of this will be discussed in more detail in the sequel to the present paper.

\section*{Acknowledgments}

The authors thank Galyna Popovych for helpful discussions and interesting comments.
The authors are also grateful to the anonymous reviewer for valuable suggestions.
This research was undertaken, in part,
thanks to funding from the Canada Research Chairs program and the NSERC Discovery Grant program.
The research of ROP was supported by the Austrian Science Fund (FWF),
project P25064, and by ``Project for fostering collaboration in science, research and education''
funded by the Moravian-Silesian Region, Czech Republic.
ROP is also grateful to the project No.\ CZ.$02.2.69\/0.0/0.0/16\_027/0008521$
``Support of International Mobility of Researchers at SU'' which supports international cooperation.
The research of AS was supported in part by the Grant Agency of the Czech Republic (GA \v{C}R)
under grant P201/12/G028, and by the RVO funding for I\v{C}47813059.

\footnotesize
\frenchspacing

\end{document}